\RequirePackage{fix-cm}
\documentclass[smallextended]{svjour3}       
\smartqed  
\usepackage{graphicx}
\usepackage{amsmath}
\usepackage{amsfonts}
\usepackage{enumitem}
\usepackage{booktabs}
\usepackage{epstopdf}
\usepackage{subcaption}
\usepackage{mathrsfs}

\newtheorem{thm}{Theorem}
\newtheorem{rem}{Remark}

\newtheorem{df}{Definition}

\newtheorem{pro}{Property}

\newcommand{\Kone}{K}
\newcommand{\Ktwo}{\mathbb{K}}

\newcommand{\Pone}{P}
\newcommand{\Ptwo}{\mathbb{P}}

\journalname{Numerical Algorithms}
\begin{document}

\title{Enabling four-dimensional conformal hybrid meshing with cubic pyramids}


\author{Miroslav S. Petrov \and
        Todor D. Todorov    \and
        Gage S. Walters     \and
        David M. Williams  \and
        Freddie D. Witherden
}


\institute{M. Petrov \at
Department of Technical Mechanics,
Technical University, 5300 Gabrovo, Bulgaria
          \and
           T. Todorov \at
Department of Mathematics, Informatics and Natural Sciences,
Technical University, 5300 Gabrovo, Bulgaria
          \and 
          G. Walters \at
Department of Mechanical Engineering, The Pennsylvania State University, University Park, Pennsylvania 16802, USA
          \and 
          D. Williams \at
Department of Mechanical Engineering, The Pennsylvania State University, University Park, Pennsylvania 16802, USA \\
\email{david.m.williams@psu.edu} 
\and
        F. Witherden \at
Department of Ocean Engineering, Texas A\&M University, College Station, Texas 77843, USA
}

\date{Received: date / Accepted: date}

\maketitle

\begin{abstract}
The main purpose of this article is to develop a novel refinement strategy for four-dimensional hybrid meshes based on cubic pyramids. This optimal refinement strategy subdivides a given cubic pyramid into a conforming set of congruent cubic pyramids and invariant bipentatopes. The theoretical properties of the refinement strategy are rigorously analyzed and evaluated. In addition, a new class of fully symmetric quadrature rules with positive weights are generated for the cubic pyramid. These rules are capable of exactly integrating polynomials with degrees up to 12. Their effectiveness is successfully demonstrated on polynomial and transcendental functions. Broadly speaking, the refinement strategy and quadrature rules in this paper open new avenues for four-dimensional hybrid meshing, and space-time finite element methods. 

\end{abstract}

\section{Introduction} \label{intro}
The classical way of solving nonstationary boundary value problems applies separate discretizations in space and time.
One promising alternative to this approach is the finite element method with full space-time discretization.  
In principle, this method can be very effective for solving hyperbolic and parabolic problems.
Some basic examples of space-time methods can be found in~\cite{steinbach2015space,steinbach2017algebraic,toulopoulos2020space} and the recent survey of Langer and Steinbach~\cite{langer2019space}. 
In addition, a representative example of this work is provided by Dumont et al.~\cite{dumont2012space} who developed a space-time finite element method to solve elastodynamics problems. 
Since the dynamics problem in~\cite{dumont2012space} involves a three-dimensional spatial body, the authors have constructed four-dimensional meshes. 
Dumont et al.~have not developed general, unstructured triangulations in four-dimensional domains. 
Rather, they firstly make a spatial triangulation and then extend the domain in the temporal direction to obtain a four-dimensional finite element triangulation. 
A similar approach for generating four-dimensional simplex meshes has been employed by Neum{\"u}ller and Karabelas in~\cite{neumuller2019generating}. A detailed survey of four-dimensional structured and unstructured simplicial meshing techniques is provided by Caplan~\cite{caplan2019four} and Frontin et al.~\cite{frontin2020polytopes}. Finally, examples of refinement strategies for high-dimensional simplicial meshes have been developed in the work of Brandts et al.~\cite{brandts2007simplicial}, and Korotov and K\v{r}\'{\i}\v{z}ek~\cite{korotov2014red}.

There have been virtually no efforts to develop four-dimensional hybrid meshes. 
However, in the last few decades, there has been interest in the development of hybrid \emph{three-dimensional} meshes~\cite{bergot2010higher,cao2012new}, with some important work being contributed by Yamakawa and coworkers~\cite{yamakawa2011subdivision,yamakawa2003increasing,yamakawa2009converting,yamakawa201088,yamakawa2011automatic}. In~\cite{yamakawa2011automatic}, Yamakawa and Shimada developed a procedure for creating hexahedron dominant meshes in order to model thin-wall structures. 
Additionally, Yamakawa et al. described a conforming coupling between hexahedron and tetrahedron meshes by using square pyramid interface elements~\cite{yamakawa2009converting}, and using triangular prism interface elements~\cite{yamakawa2011subdivision}. 
The conforming coupling of hexahedra and tetrahedra via pyramids has been discussed independently by Meshkat and Talmor~\cite{meshkat2000generating}, and Devloo et al.~\cite{devloo2019high}. 
In addition, Pantano and Averill~\cite{pantano2007penalty} have used a non-conforming, penalty-based procedure to couple independently modeled three-dimensional finite element meshes. 
Similar non-conforming techniques using mortar elements have been explored by Maday et al.~\cite{maday1988nonconforming} and Seshaiyer and Suri~\cite{seshaiyer2000hp}. 
Furthermore, hybrid non-conforming hexahedral-tetrahedral meshes have been examined by Reberol and L{\'e}vy~\cite{reberol2016low}.
Unfortunately, these non-conforming meshes are incompatible with most standard finite element methods, with the exception of discontinuous Galerkin methods.
Evidently, the most convenient way to unify independently created finite element meshes is through the conforming coupling of the elements in the interface subdomain. This strongly motivates the importance of transitional elements, such as pyramidal and prismatic elements.

Let us briefly review previous efforts to analyze transitional elements. 
Evidently, the three-dimensional prismatic elements are relatively straightforward to analyze, as they can be constructed by forming tensor products between line segments and triangles, whereas, a similar procedure is impossible for pyramids. 
As a result, prismatic elements inherit their interpolation and quadrature procedures from well-established 1D/2D procedures, and researchers have primarily focused their attention on pyramids. 
The pioneering work on pyramids was performed by Bedrosian~\cite{bedrosian1992shape} in the early nineties. 
He was the first researcher who explained the role of pyramidal elements for coupling different kinds of meshes, and developed the first interpolation and integration procedures. 
A number of researchers have followed in Bedrosian's footsteps~\cite{ainsworth2017lowest,bergot2010higher,chan2015comparison,chan2015hp,chan2016orthogonal,chan2016short,coulomb1997pyramidal,gillette2016serendipity}. In particular, Bergot et al.~\cite{bergot2010higher} performed a deep analysis of the interpolation properties of pyramidal elements. 
Chan and Warburton extended this work, developing alternative sets of basis functions~\cite{chan2016orthogonal,chan2016short}, analyzing the numerical stability of interpolation points~\cite{chan2015comparison}, and deriving trace inequalities~\cite{chan2015hp}. 
Furthermore, Gillette~\cite{gillette2016serendipity} recently used techniques from finite element exterior calculus~\cite{arnold2010finite,arnold2006finite} to construct seredipity-type basis functions on pyramids. 
Quadrature formulas on pyramidal elements have been created and analyzed in~\cite{bergot2010higher,nigam2012numerical,witherden2015identification}.
The error estimates arising from the quadrature rules of Bergot et al.~\cite{bergot2010higher} have been improved by Nigam and Phillips~\cite{nigam2012numerical} in order to obtain the classical rate of convergence for second-order boundary value problems. However, the best set of rules to date (in terms of integration strength for a minimum number of points) was identified by Witherden and Vincent in~\cite{witherden2015identification}.

Now, let us return our attention to four-dimensional space. It turns out that the three-dimensional definitions of transitional elements do not extend to higher dimensions.
For example, while the triangular prism provides a conforming interface between the hexahedron and the tetrahedron in three dimensions, its analog in four dimensions (the tetrahedral prism) does \emph{not} provide a conforming interface between the tesseract and the pentatope. 
This immediately follows from the fact that the tetrahedral prism has two  tetrahedral facets and four triangular prismatic facets, and therefore it cannot conformally interface with the tesseract. In a similar fashion, the cubic pyramid (the four-dimensional analog of the square pyramid) has one hexahedral facet and six square pyramidal facets, and therefore it cannot conformally interface with the pentatope. Fortunately, the lack of transitional elements between tesseracts and pentatopes does not completely prevent the formulation of hybrid meshes in four dimensions. We will show that it is still possible to build hybrid meshes using tesseract, cubic pyramid, and pentatopal elements. We note that some basic geometric properties of these elements are discussed in the work of Coxeter~\cite{coxeter1940regular,coxeter1973regular}, Sommerville~\cite{mclaren1958introduction}, and Zamboj~\cite{zamboj2018sections}.

The main goals of this paper are: i) to develop conforming hybrid four-dimensional meshes, ii) to create an optimal refinement strategy for these meshes, and iii) to develop numerical integration procedures on the elements of these meshes.
The major contributions of the paper are summarized as follows.
First, we identify a four-dimensional conforming coupling between tesseract elements and cubic pyramid elements.
Next, the tesseract and cubic pyramid elements are subdivided, while maintaining conformity.
Evidently, while the tesseracts can be uniformly subdivided into smaller tesseracts, the cubic pyramids cannot be subdivided into smaller cubic pyramids without leaving gaps. 
Therefore, in the refinement strategy for the cubic pyramids, we use a combination of congruent cubic pyramids \emph{and} invariant bipentatopes.
A simple two-level refinement tree is obtained. 
The theoretical properties of the refinement strategy are thoroughly analyzed. The proposed theoretical results cannot be improved, otherwise the cubic pyramid would be invariant with a one-level refinement tree, which is obviously not the case. 
Finally, we conclude our work by developing numerical integration procedures for the cubic pyramid elements. 
We note that, while such procedures are already well-established for the tesseract and pentatope elements~\cite{frontin2020polytopes,williams2020family}, they have yet to be developed for cubic pyramid elements.

The format of this paper is as follows. 
In Section 2, we introduce some standard notation and terminology. 
In Section 3, we formulate new hybrid meshes of tesseract, cubic pyramid, and bipentatope elements, along with a non-degenerate mesh refinement strategy. 
In Section 4, we develop theoretical results which govern the hybrid meshes. 
In Section 5, we introduce a new set of fully symmetric quadrature rules for cubic pyramid elements. 
Finally, in Section 6, we summarize the main conclusions of the paper.

\section{Preliminary concepts}

\subsection{Background and motivation}

The four-dimensional space $\mathbb{R}^{4}$ has unique geometric properties which cannot be found in other multidimensional Euclidean spaces excluding the two-dimensional one.
This fact can be illustrated by comparing the three- and four-dimensional spaces.
It is well-known that the 3-cube cannot be triangulated by standard tetrahedra (i.e.~cube corners)~\cite{korotov2014red}. 
Rather, the cube is usually triangulated into five simplices, four of which are standard tetrahedra and one of which is a regular tetrahedron. 
Additionally, it is impossible to triangulate the standard tetrahedron with standard tetrahedra. 
Instead, there exists a partition of the standard tetrahedron into six standard tetrahedra and one regular tetrahedron~\cite{todorov2013optimal}.
Interestingly enough, these principles do not extend to the four-dimensional case. Petrov and Todorov~\cite{petrov2018stable} have proved that each tesseract can be divided into standard pentatopes (i.e.~tesseract corners), and each standard pentatope can be partitioned into standard pentatopes.
The latter fact means that an arbitrary 4D \emph{canonical} domain can be divided into standard pentatopes~\cite{petrov2018stable}, which are the most convenient elements from a computational point of view.
The authors have tested six- and eight-dimensional hypercubes for such properties, but without success. Therefore, given its useful properties, more attempts to explore and analyze four-dimensional space are certainly warranted.

In what follows, we will introduce some useful definitions which facilitate our subsequent discussions of $\mathbb{R}^{4}$.

\subsection{Definitions}

\begin{df}
A  simply connected  domain in $\mathbb{R}^{4}$ is called  canonical if there exists a conforming partition of the domain into tesseracts (4-cubes).
\end{df}

\begin{df}
Each $n$-dimensional hypercube can be divided into $2n$ $(n-1)$-hypercube pyramids. These pyramids are called canonical pyramids.
\end{df}

\begin{df}
Let $\Omega$  be a nondegenerated polytope in $\mathbb{R}^4$, and $T_i$, $i=1,2,\ldots,k,$ be $4$-dimensional finite elements. The triangulation 
\begin{align*}
\tau=\left\{T_i\subset \mathbb{R}^4 \ | \ \Omega=\bigcup_{i=1}^k T_i\right\},
\end{align*}
of the polytope $\Omega$ is said to be consistent if $T_i$ and $T_j$, $1\le i<j\le k$, share nothing other than the empty set, a vertex, an edge  or an $\ell$-dimensional facet, $\ell=2,3$.
\end{df}

\begin{df}
The polytope
\begin{align*}
T_{-i}=[t_{1},t_{2},\dots,t_{i-1},t_{i+1},\dots,t_{n+1}],\quad
i=2,\dots,n,
\end{align*}
related to the vertex $t_i$ is  obtained from the polytope  $T=[t_{1},t_{2},\dots,t_{n+1}]$, $n\in \mathbb{N}$ by removing the vertex $t_i$.
The polytopes $T_{-1}$ and $T_{-(n+1)}$ are defined in a similar fashion as
\begin{align*}
T_{-1}=[t_{2},t_{3},\dots,t_{n+1}],\quad
T_{-(n+1)}=[t_{1},t_{2},\dots,t_{n}].
\end{align*}
\end{df}

\begin{df}
The polytopes $T_1$ and  $T_2$ are congruent if one of
them can be obtained from the other by applying a linear transformation
\begin{align*}
T_2 =\underline{b}+cQT_1,
\end{align*}
where $c\not=0$ is a scaling factor, $\underline{b}$ is a translation vector,
and  $Q$ is an orthogonal matrix.
\end{df}

\begin{df}
We say that two polytopes  $T_1$ and $T_2$ are from the same class  if  they are congruent.
\end{df}
The class
\begin{align*}
[K]=\{T\subset \mathbb{R}^4 \ | \ T\cong K\},
\end{align*}
consists of all equivalent pyramids to the pyramid $K$ in $\mathbb{R}^4$ with respect to the congruence relation.

\begin{df}
A polytope $T$ is said to be invariant concerning a refinement strategy $\mathcal{A}$ if all elements of $\mathcal{A}T$ belong to $[T]$.
\end{df}

\begin{df}
The degeneracy measure  of an arbitrary pyramid
$K = [k_1,k_2,$ $\ldots,$ $k_{n+1}]$ is equal to
\begin{equation*}
\delta(K)= \frac{h(K)\cdot {\rm vol}(\partial K)}{8\cdot {\rm vol}(K)},
\end{equation*}
where $h(K)$ is the diameter of $K$.
\end{df}
Here, we use ${\rm vol}(T)$ and ${\rm vol}(\partial T)$
instead of ${\rm vol}_{4}(T)$ and ${\rm vol}_{3}(\partial T)$ in order to avoid complicated notation.

\subsection{Reference elements}

In this section, we introduce a set of convenient reference elements in $\mathbb{R}^4$. 
We review the  important properties of these elements, including their vertex locations, integration limits, and orthonormal polynomial bases. 

In order to prepare our discussion of the reference elements, let us briefly review the definition of an orthonormal basis. Broadly speaking, the basis maintains the following key property
\begin{equation*}
    \int_{\Omega} \psi_{ijkq} \left( \boldsymbol{x} \right) \psi_{rstv} \left( \boldsymbol{x} \right) d \boldsymbol{x} = \delta_{ir} \delta_{js} \delta_{kt} \delta_{qv},
\end{equation*}
where $\delta_{ir}$ is the Kronecker delta. 
The orthonormal basis functions of degree $p$ have the form
\begin{equation*}
    \psi_{ijkq} \left( \boldsymbol{x} \right) = \zeta_{ijkq} \, \hat{P}_i^{(\alpha_1 ,\beta_1)} \left( a \right) \hat{P}_j^{(\alpha_2 ,\beta_2)} \left( b \right) \hat{P}_k^{(\alpha_3 ,\beta_3)} \left( c \right) \hat{P}_q^{(\alpha_4 ,\beta_4)} \left( d \right) f\left( \boldsymbol{x} \right),
\end{equation*}
where $i + j +k +q \leq p$, $a = a\left(x_1, x_4\right)$, $b = b\left(x_2, x_4\right)$, $c = c\left(x_3, x_4\right)$, $d = d\left(x_4\right)$, and $f = f\left( \boldsymbol{x}\right)$ are functions depending on the element type, $\zeta_{ijkq}$, $\alpha_1 ,\ldots, \alpha_4$ and $\beta_1, \ldots, \beta_4$ are constants, and $\hat{P}_n^{(\alpha ,\beta)}$ are the 1D orthonormal Jacobi polynomials defined as
\begin{equation*}
    \hat{P}_n^{(\alpha,\beta)} \left( x_1 \right) = \frac{P_n^{(\alpha,\beta)} (x_1)}{\sqrt{\frac{2^{\alpha+\beta+1}}{2n + \alpha + \beta + 1}\frac{ \left( n + \alpha \right)! \left( n + \beta \right)! }{ n! \left( n + \alpha + \beta \right)!}}} .
\end{equation*}
Here, the functions $P_n^{(\alpha,\beta)}$ are the well-known orthogonal Jacobi polynomials, which themselves are \emph{not} orthonormal. The scale factor under the square root operator provides the desired normalization. 

It is convenient to omit the Jacobi polynomial superscripts when $\alpha = \beta = 0$. Consequently, we frequently write $\hat{P}_{n}$ in place of $\hat{P}_{n}^{\left(0,0\right)}$.

Now, let us turn our attention to the definitions of the reference elements.

\subsubsection{Tesseract}

Consider the reference tesseract $T^{\ast}$ centered at the origin, having edge length $l[T^{\ast}] = 2$.
The vertices are defined such that
\begin{align*}
  T^{\ast} = \Big[ &t_1\left(-1, -1, -1, -1\right), \; t_2\left(-1, 1, -1, -1\right), \; t_3\left(-1, -1, -1, 1\right), \; t_4\left(1, -1, -1, -1\right), \\
  &t_5\left(-1, 1, -1, 1\right), \; t_6\left(1, -1, -1, 1\right), \;
  t_7\left(1, 1, -1, -1\right), \;
  t_8\left(1, 1, -1, 1\right), \\
  &t_9\left(-1, -1, 1, 1\right), \;
  t_{10}\left(-1, 1, 1, 1\right), \;
  t_{11}\left(-1, -1, 1, -1\right), \;
  t_{12}\left(1, -1, 1, 1\right), \\
  &t_{13}\left(-1, 1, 1, -1\right), \;
  t_{14}\left(1, -1, 1, -1\right), \;
  t_{15}\left(1, 1, 1, 1\right), \;
  t_{16}\left(1, 1, 1, -1\right) \Big],
\end{align*}
as shown in Figure~\ref{tesseract}. 
\begin{figure}[h!]
\begin{center}
\includegraphics[width=10cm]{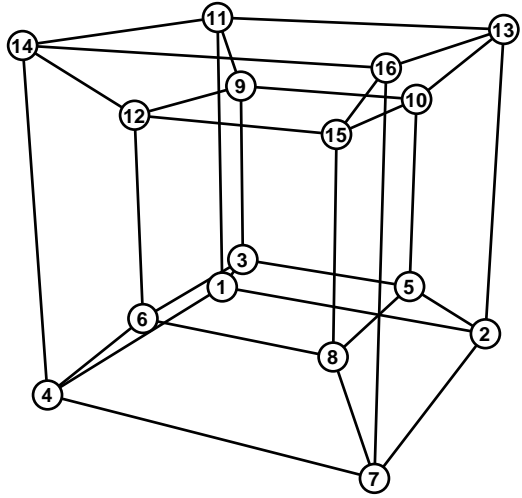}
\end{center}
\caption{The reference tesseract $T^{\ast}$.}
\label{tesseract}
\end{figure}
The volume is $\text{vol}\left( T^{\ast} \right) = 16$ and it is straightforward to define bounds of integration
\begin{equation*}
    \int_{-1}^1 \int_{-1}^1 \int_{-1}^1 \int_{-1}^1 dx_1 dx_2 dx_3 dx_4 = 16.
\end{equation*}
Trivially, the orthonormal basis inside of the reference tesseract is given by
\begin{equation*}
    \psi_{ijkq} \left( \boldsymbol{x} \right) =  \hat{P}_i \left( a \right) \hat{P}_j \left( b \right) \hat{P}_k \left( c \right) \hat{P}_q \left( d \right),
\end{equation*}
where $a = x_1$, $b = x_2$, $c = x_3$, and $d = x_4$.

\subsubsection{Cubic pyramid}

Consider the reference cubic pyramid 
$K^{\ast}$ centered at the origin, having edge length $l[K^{\ast}] = 2$. 
The vertices are defined such that
\begin{align*}
    K^{\ast} = \Big[ &k_1\left(-1, -1, -1, -1\right), \; k_2\left(-1, 1, -1, -1\right), \; k_3\left(-1, -1, 1, -1\right), \; k_4\left(1, -1, -1, -1\right), \\
    & k_5\left(-1, 1, 1, -1\right), \; k_6\left(1, -1, 1, -1\right), \;
    k_7\left(1, 1, -1, -1\right), \;
    k_8\left(1, 1, 1, -1\right), \\
    &k_9\left(0, 0, 0, 0\right) \Big],
\end{align*}
as shown in Figure~\ref{cubic_pyramid}.
\begin{figure}[h!]
\begin{center}
\includegraphics[trim = 0cm 2.5cm 0cm 2cm, clip,width=10cm]{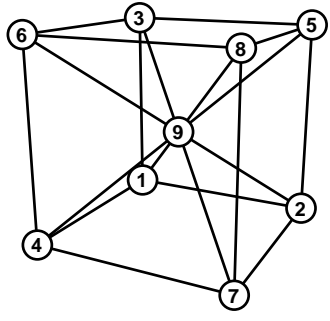}
\end{center}
\caption{The reference cubic pyramid $K^{\ast}$.}
\label{cubic_pyramid}
\end{figure}
The volume of the cubic pyramid is $1/8$th of the tesseract, or equivalently $\text{vol}(K^{\ast}) = \frac{A h}{4}$ where $A = 8$ is the area of the hexahedron base and $h = 1$ is the height of the cubic pyramid. 
Based on the vertices of $K^{\ast}$, the integration bounds for the cubic pyramid are
\begin{equation*}
    \int_{-1}^{0} \int_{x_4}^{-x_4} \int_{x_4}^{-x_4} \int_{x_4}^{-x_4} \,\mathrm{d}x_1 \, \mathrm{d}x_2\, \mathrm{d}x_3 \, \mathrm{d}x_4 = 2.
\end{equation*}
The orthonormal polynomial basis for the cubic pyramid is found to be
\begin{equation}
    \label{eq:cubpb}
    \psi_{ijkq} \left( \boldsymbol{x} \right) = \sqrt{2^{\mu_{ijk}+1}} \hat{P}_i \left( a \right) \hat{P}_j \left( b \right) \hat{P}_k \left( c \right) \hat{P}_q^{(\mu_{ijk},0)} \left( d \right) \left(-x_4 \right)^{i+j+k},
\end{equation}
where $a = -\frac{x_1}{x_4}$, $b=-\frac{x_2}{x_4}$, $c = -\frac{x_3}{x_4}$, $d=2x_4 +1$, and $\mu_{ijk} = 2i + 2j + 2k + 3$.

\subsubsection{Bipentatope}

Consider the reference bipentatope
$P^{\ast}$ having edge length $l[P^{\ast}] = 1$. 
This element is also frequently called the tetrahedral bipyramid. 
The vertices are defined such that
\begin{align*}
    P^{\ast} = \Big[ &p_1\left(\tfrac{1}{2}, \tfrac{1}{2}, \tfrac{1}{2}, -\tfrac{1}{2}\right), \; p_2\left(\tfrac{1}{2}, \tfrac{1}{2}, -\tfrac{1}{2}, -\tfrac{1}{2}\right), \; p_3\left(-\tfrac{1}{2}, \tfrac{1}{2}, -\tfrac{1}{2}, -\tfrac{1}{2}\right), \; p_4\left(-\tfrac{1}{2}, \tfrac{1}{2}, \tfrac{1}{2}, -\tfrac{1}{2}\right), \\
    &p_5\left(0, 1, 0, -1\right), \; p_6\left(0, 0, 0, -1\right) \Big],
\end{align*}
as shown in Figure~\ref{bipentatope}. 
\begin{figure}[h!]
\begin{center}
\includegraphics[trim = 0cm 0.5cm 0cm 0.5cm, clip,width=10cm]{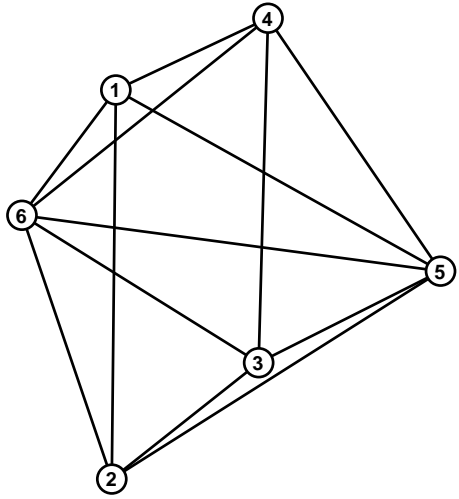}
\end{center}
\caption{The reference bipentatope $P^{\ast}$.}
\label{bipentatope}
\end{figure}
The reference bipentatope volume is $\text{vol}\left(P^{\ast}\right) = \tfrac{1}{24}$. 

In order to facilitate interpolation and integration, the reference bipentatope can be divided into two pentatopes, as shown in Figure~\ref{split_bipentatope}.
\begin{figure}[h!]
\begin{center}
\includegraphics[width=7cm]{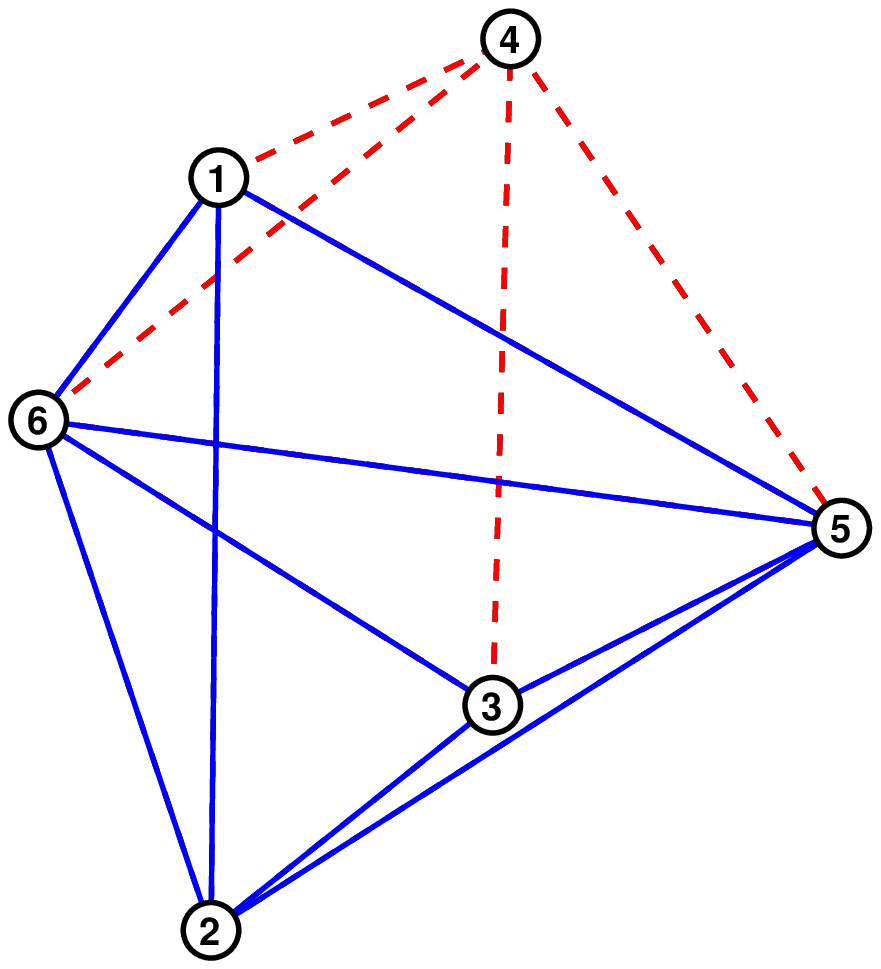}
\begin{subfigure}[b]{0.45\textwidth}
\includegraphics[width=6cm]{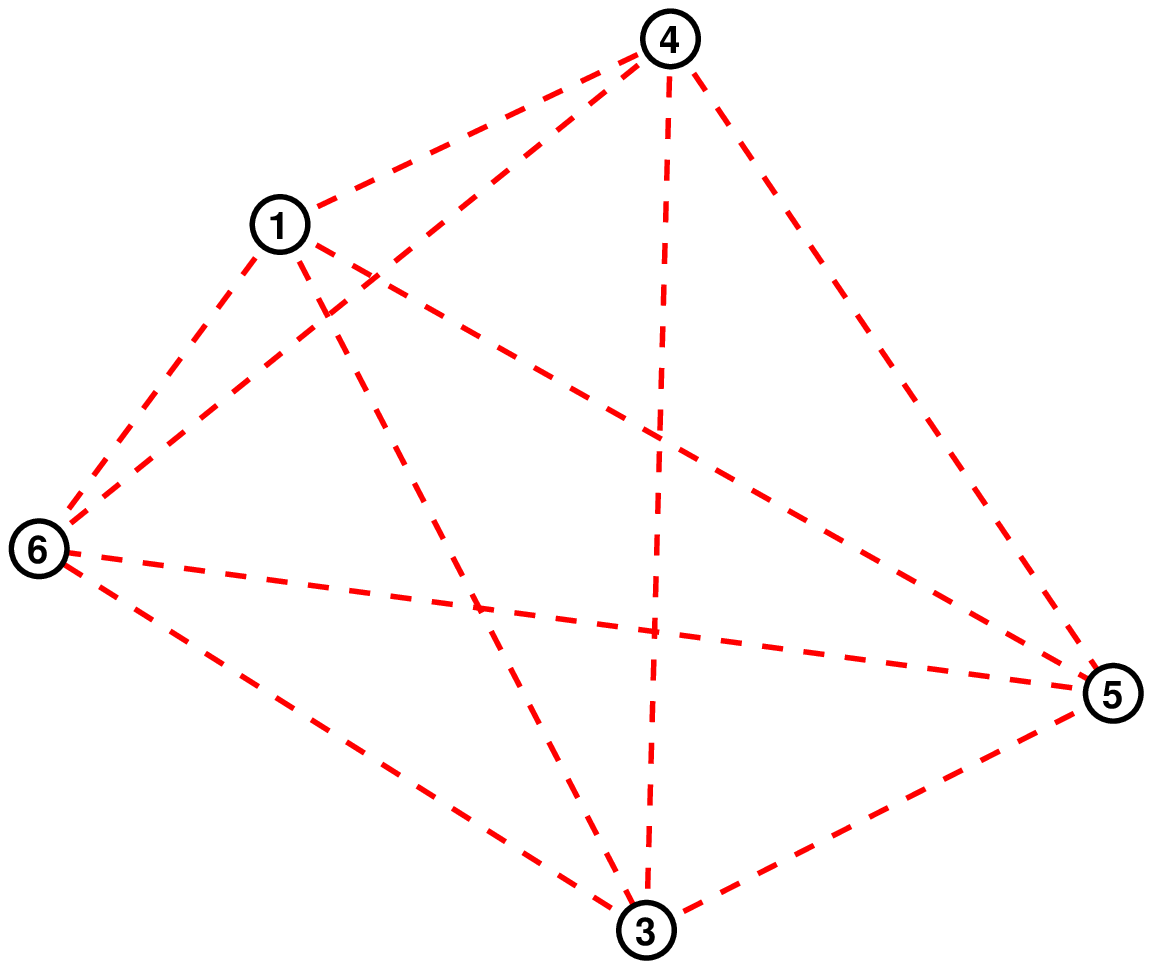}
\end{subfigure}
\hfill
\begin{subfigure}[b]{0.45\textwidth}
\includegraphics[width=6cm]{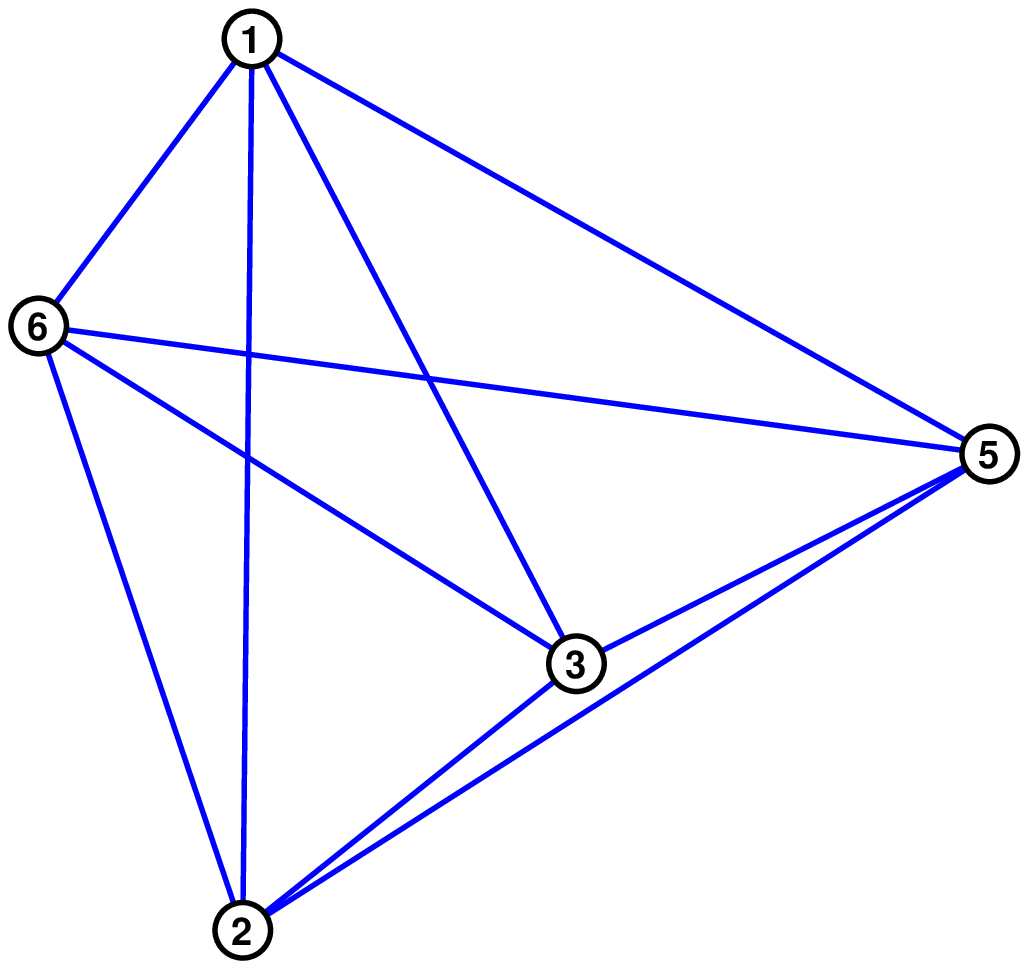}
\end{subfigure}
\end{center}
\caption{A subdivision of the reference bipentatope into two pentatopes.}\label{split_bipentatope}
\end{figure}
There exists a convenient mapping from these pentatopes onto the standard pentatope $S^{\ast}$ shown in Figure~\ref{standard_pentatope}.
\begin{figure}[h!]
\begin{center}
\includegraphics[trim = 0cm 0.5cm 0cm 0.5cm, clip,width=10cm]{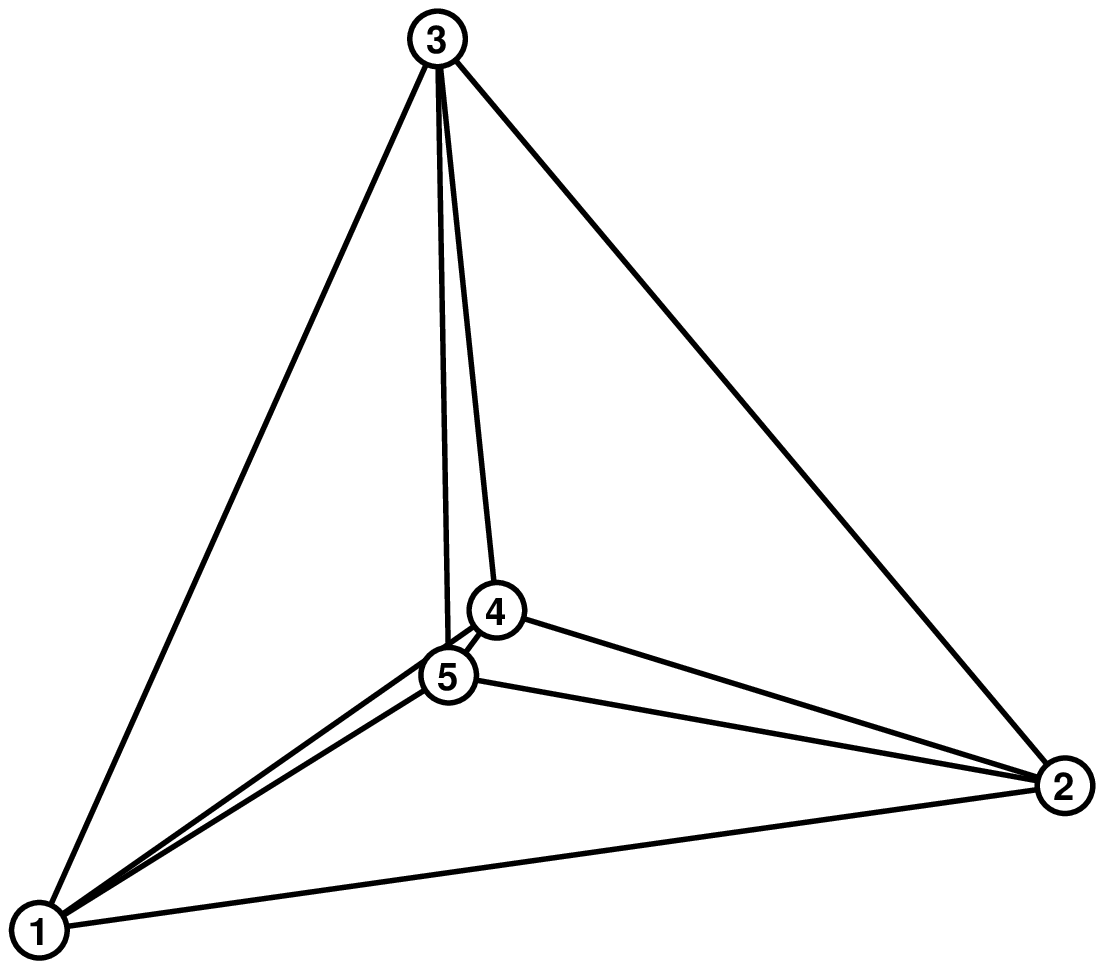}
\end{center}
\caption{The standard pentatope $S^{\ast}$.}
\label{standard_pentatope}
\end{figure}
We note that $S^{\ast}$ has the following vertices
\begin{align*}
    S^{\ast} = &\Big[s_1 \left(1,-1,-1,-1\right), s_2 \left(-1,1,-1,-1\right), s_3\left(-1,-1,1,-1\right), s_4 \left(-1,-1,-1,1\right), \\
    &s_5\left(-1,-1,-1,-1\right) \Big].
\end{align*}
The standard pentatope volume is $\text{vol} \left(S^{\ast}\right) = \tfrac{2}{3}$ and the bounds of integration are
\begin{equation*}
    \int_{-1}^1 \int_{-1}^{-1-x_4} \int_{-1}^{-1-x_3 - x_4} \int_{-1}^{-1-x_2 - x_3 -x_4 } dx_1 dx_2 dx_3 dx_4 = \frac{2}{3}.
\end{equation*}
The orthonormal polynomial basis takes the following form
\begin{align}
    \nonumber \psi_{ijkq} \left( \boldsymbol{x} \right) = &4 \hat{P}_i \left( a \right) \hat{P}_j^{(2i+1,0)} \left( b \right) \hat{P}_k^{(2i+2j+2,0)} \left( c \right) \hat{P}_q^{(2i+2j+2k+3,0)} \left( d \right) \\
    &\times \left( 1-b \right)^{i} \left( 1-c \right)^{i + j} \left( 1-d \right)^{i + j +k}, \label{pent_basis}
\end{align}
where $a = -2\frac{x_1 +1}{x_2 +x_3 +x_4 +1} - 1$, $b=-2\frac{1+x_2}{x_3 +x_4} - 1$, $c = 2\frac{1+x_3}{1-x_4} - 1$, and $d=x_4$.

Lastly, we consider the following rescaled and shifted version of the standard pentatope
\begin{align*}
\mathcal{S} ^{\ast} = \left[
s_1^*(1,0,0,0),s_2^*(0,1,0,0),s_3^*(0,0,1,0),
s_4^*(0,0,0,1),s_5^*(0,0,0,0)\right],
\end{align*}
where by inspection
\begin{align*}
    \mathcal{S}^{\ast} = \tfrac{1}{2} S^{\ast} + \left(1,1,1,1\right).
\end{align*}
This reference element has been used by many authors~\cite{brandts2007simplicial,petrov2018stable} due to its mathematical simplicity and convenience. It is frequently referred to as the `cube corner'. Its volume is $\text{vol}\left( \mathcal{S}^{\ast} \right) =\tfrac{1}{24}$ and the bounds of integration are
\begin{equation*}
    \int_{0}^1 \int_{0}^{1-x_4} \int_{0}^{1-x_3 - x_4} \int_{0}^{1-x_2 - x_3 -x_4} dx_1 dx_2 dx_3 dx_4 = \frac{1}{24}.
\end{equation*}
The orthonormal polynomial basis on the cube corner can be obtained by a simple linear transformation of the basis in Eq.~\eqref{pent_basis}.

\pagebreak
\clearpage

\section{Hybrid meshes of tesseracts, cubic pyramids, and bipentatopes}

In this section, we introduce a set of hybrid meshes by constructing the following domain model.
Let $\Omega$ be a four-dimensional canonical domain. The domain $\Omega$ is partitioned into two subdomains $\Omega_1$ and $\Omega_2$, which are also canonical. 
We suppose that $\check{\tau}_0$ and $\hat{\tau}_0$ are tesseract triangulations of $\Omega_1$ and $\Omega_2$ such that all elements of both triangulations form a conforming tesseract triangulation $\tau_0$ of the domain $\Omega$.
We define the boundary layers in the subdomains $\Omega_1$ and $\Omega_2$ by
\begin{align*}
B_1&=\{T\in\Omega_1 \ | \
\dim\left(\partial T\cap\Omega_{12}\right)=3\}, \\[1.5ex]
B_2&=\{T\in\Omega_2 \ | \
\dim\left(\partial T\cap\Omega_{12}\right)=3\},\quad
\Omega_{12}=\Omega_{1}\cap\Omega_{2}.
\end{align*}
Furthermore, we construct a hybrid mesh so that $\Omega_1$ is partitioned by tesseract elements and $\Omega_2$ is partitioned by cubic pyramid elements.
Our main goal is to establish a conforming coupling between the different kinds of elements in the boundary layers of both subdomains. 
Additionally, we need a refinement strategy for both groups of elements with as small as possible number of classes and as small as possible measure of degeneracy for the pyramidal elements. 
We restrict ourselves to refinement strategies dividing the edges of each element into two parts. We anticipate that such strategies will be very useful in multigrid and/or adaptive-quadrature procedures. 

We will now define a refinement algorithm for creating nested hierarchical hybrid triangulations in five steps:

\begin{itemize}
    \item[$\left(r_1\right)$] The tesseract triangulations of $\Omega$, denoted by $\check{\tau}_0$, $\hat{\tau}_0$ and $\tau_0=\check{\tau}_0\cup\hat{\tau}_0$ are constructed. We denote this partition operator by $\mathcal{R}$.
    
    \item[$\left(r_2\right)$] Each element $T$ of $\check{\tau}_{0}$ is uniformly divided into sixteen smaller tesseracts. We denote this partition operator by $\mathcal{E}$. Evidently, this operator's definition is straightforward, and does not need to be explicitly stated.
    
    \item[$\left(r_3\right)$] Each element $T$ of $\hat{\tau}_0$ is divided into eight cubic pyramids by connecting the cubic facets of each element to its centroid. 
    We denote this partition operator by $\mathcal{B}$. It is given explicitly by
    \begin{align*}
    \mathcal{B} T^{\ast}
    = \Big\{ &\Kone_{1} [t_{1},t_{2},t_{3},t_{4},t_{5},t_{6},t_{7},t_{8},t_{0}],\ 
    \Kone_{2} [t_{1},t_{3},t_{4},t_{6},t_{9},t_{11},t_{12},t_{14},t_{0}], \\
    &\Kone_{3}[t_{9},t_{10},t_{11},t_{12},t_{13},t_{14},t_{15},t_{16},t_{0}],\ 
    \Kone_{4} [t_{2},t_{5},t_{7},t_{8},t_{10},t_{13},t_{15},t_{16},t_{0}], \\
    &\Kone_{5}[t_{4},t_{6},t_{7},t_{8},t_{12},t_{14},t_{15},t_{16},t_{0}], \ 
    \Kone_{6} [t_{1},t_{2},t_{3},t_{5},t_{9},t_{10},t_{11},t_{13},t_{0}], \\
    &\Kone_{7}[t_{1},t_{2},t_{4},t_{7},t_{11},t_{13},t_{14},t_{16},t_{0}], \
    \Kone_{8} [t_{3},t_{5},t_{6},t_{8},t_{9},t_{10},t_{12},t_{15},t_{0}]
    \Big\}.
    \end{align*}
    One of the resulting cubic pyramids $\Kone_7$ is shown in Figure~\ref{cubic_pyramid_in_tesseract}. Note: $\Kone_7$ is equivalent to $K^{\ast}$, which is a canonical pyramid.
    
    \item[$\left(r_4\right)$] The elements of $\mathcal{B} \hat{\tau}_0$ are further subdivided by the partition operator $\mathcal{L}$, which is defined below.
    \item[$\left(r_5\right)$] Steps 2--4 are performed repeatedly until the desired level of mesh refinement is obtained. These refinement steps are denoted by the operator $\mathcal{H}$.
\end{itemize}

\begin{figure}[h!]
\begin{center}
\includegraphics[width=9cm]{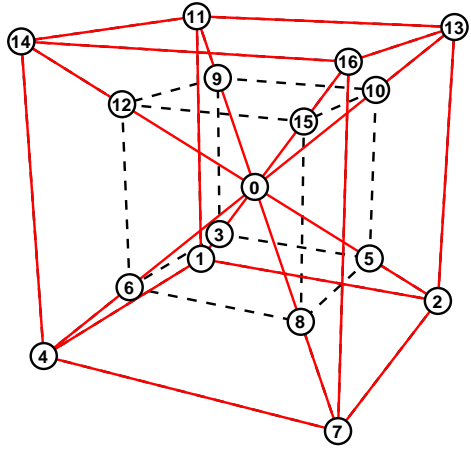}
\end{center}
\caption{The cubic pyramid $\Kone_7$ belonging to the decomposition $\mathcal{B}T^{\ast}$.}
\label{cubic_pyramid_in_tesseract}
\end{figure}

Now, let us turn our attention to formulating an explicit definition for the partition operator $\mathcal{L}$
\begin{align*}
\mathcal{L} X=
\begin{cases}
  \mathcal{D} X, & \mbox{if } X \;\mbox{is a cubic pyramid}\\
  \mathcal{M} X, & \mbox{if } X \; \mbox{is a bipentatope.}
\end{cases}
\end{align*}
It remains for us to define the operators $\mathcal{D}$ and $\mathcal{M}$. The operator $\mathcal{D}$ can be defined as follows
\begin{align*}
    \mathcal{D} K^*= \bigg\{\Big\{ &\Ktwo_1[k_{1},k_{10},k_{11},k_{12},k_{13},k_{14},k_{15},k_{16},k_{21}], \; \Ktwo_2[k_{2},k_{10},k_{13},k_{14},k_{15},k_{17},k_{18},k_{19},k_{20}], \\
    &\Ktwo_3[k_{4},k_{12},k_{14},k_{15},k_{21},k_{26},k_{27},k_{28},k_{29}], \; \Ktwo_4[k_{7},k_{14},k_{15},k_{18},k_{19},k_{27},k_{28},k_{34},k_{35}], \\
    & \Ktwo_5[k_{3},k_{11},k_{13},k_{15},k_{21},k_{22},k_{23},k_{24},k_{25}], \; \Ktwo_6[k_{5},k_{13},k_{15},k_{17},k_{19},k_{22},k_{24},k_{30},k_{31}], \\
    & \Ktwo_7[k_{6},k_{15},k_{21},k_{23},k_{24},k_{26},k_{28},k_{32},k_{33}], \; \Ktwo_8[k_{8},k_{15},k_{19},k_{24},k_{28},k_{30},k_{32},k_{34},k_{36}], \\
    & \Ktwo_9[k_{9},k_{16},k_{20},k_{25},k_{29},k_{31},k_{33},k_{35},k_{36}], \; \Ktwo_{10}[k_{15},k_{16},k_{20},k_{25},k_{29},k_{31},k_{33},k_{35},k_{36}] \Big\}, \\[-0.5ex]
    \Big\{ &\Pone_{1}[k_{15},k_{21},k_{23},k_{24},k_{25},k_{33}], \; \Pone_{2}[k_{11},k_{13},k_{15},k_{16},k_{21},k_{25}], \; \Pone_{3}[k_{12},k_{14},k_{15},k_{16},k_{21},k_{29}], \\
    & \Pone_{4}[k_{15},k_{21},k_{26},k_{28},k_{29},k_{33}], \;
    \Pone_{5}[k_{15},k_{24},k_{28},k_{32},k_{33},k_{36}], \; \Pone_{6}[k_{14},k_{15},k_{27},k_{28},k_{29},k_{35}], \\
    & \Pone_{7}[k_{15},k_{19},k_{28},k_{34},k_{35},k_{36}],\; \Pone_{8}[k_{13},k_{15},k_{17},k_{19},k_{20},k_{31}], \;
    \Pone_{9}[k_{15},k_{19},k_{24},k_{30},k_{31},k_{36}], \\ 
    & \Pone_{10}[k_{14},k_{15},k_{18},k_{19},k_{20},k_{35}], \;
    \Pone_{11}[k_{13},k_{15},k_{22},k_{24},k_{25},k_{31}], \; \Pone_{12}[k_{10},k_{13},k_{14},k_{15},k_{16},k_{20}], \\
    & \Pone_{13}[k_{15},k_{19},k_{20},k_{31},k_{35},k_{36}],\; \Pone_{14}[k_{15},k_{28},k_{29},k_{33},k_{35},k_{36}], \;
    \Pone_{15}[k_{15},k_{16},k_{21},k_{25},k_{29},k_{33}], \\[-1.2ex] 
    & \Pone_{16}[k_{13},k_{15},k_{16},k_{20},k_{25},k_{31}], \;
    \Pone_{17}[k_{15},k_{24},k_{25},k_{31},k_{33},k_{36}],\; \Pone_{18}[k_{14},k_{15},k_{16},k_{20},k_{29},k_{35}]\Big\} \bigg\}.
\end{align*}
All elements $\Ktwo_i$, $i=1,2,\ldots,10$, are cubic pyramids
and all elements $\Pone_j$, $j=1,2,\ldots,18$, are bipentatopes.
The first nine vertices (above) $k_1, k_2, \ldots, k_9$ are previously defined in the definition of $K^{\ast}$. 
In addition, the remaining vertices $k_{10}, \ldots, k_{36}$ can be defined as follows
\begin{alignat*}{4}
    &k_{10} = \tfrac{1}{2}(k_1+k_2), \qquad &&k_{11} = \tfrac{1}{2}(k_1+k_3), \qquad &&k_{12} = \tfrac{1}{2}(k_1+k_4), \qquad &&k_{13} = \tfrac{1}{2}(k_{1}+k_{5}), \\
    &k_{14} = \tfrac{1}{2}(k_{1}+k_{7}), \; &&k_{15} = \tfrac{1}{2}(k_{1}+k_{8}), \; &&k_{16} = \tfrac{1}{2}(k_1+k_9), \; &&k_{17} = \tfrac{1}{2}(k_2+k_5), \\
    & k_{18} = \tfrac{1}{2}(k_2+k_7), \; &&k_{19} = \tfrac{1}{2}(k_{2}+k_{8}), \; &&k_{20} = \tfrac{1}{2}(k_2+k_9), \; &&k_{21} = \tfrac{1}{2}(k_{3}+k_{4}), \\
    & k_{22} = \tfrac{1}{2}(k_3+k_5), \; &&k_{23} = \tfrac{1}{2}(k_3+k_6), \; &&k_{24} = \tfrac{1}{2}(k_{3}+k_{8}), \; &&k_{25} = \tfrac{1}{2}(k_3+k_9), \\
    & k_{26} = \tfrac{1}{2}(k_4+k_6), \; &&k_{27} = \tfrac{1}{2}(k_4+k_7), \; &&k_{28} = \tfrac{1}{2}(k_{4}+k_{8}), \; &&k_{29} = \tfrac{1}{2}(k_4+k_9), \\
    & k_{30} = \tfrac{1}{2}(k_5+k_8), \; &&k_{31} = \tfrac{1}{2}(k_5+k_9), \; &&k_{32} = \tfrac{1}{2}(k_6+k_8), \; &&k_{33} = \tfrac{1}{2}(k_6+k_9), \\
    & k_{34} = \tfrac{1}{2}(k_7+k_8), \; &&k_{35} = \tfrac{1}{2}(k_7+k_9), \; &&k_{36} = \tfrac{1}{2}(k_8+k_9).
\end{alignat*}
The full set of vertices is illustrated in Figure~\ref{cubic_pyramid_with_vertices}.
In addition, a few of the individual elements of $\mathcal{D} K^{\ast}$ are shown in Figure~\ref{sample_elements_cubic_pyramid}. 
Finally, several couplings between the cubic pyramid and bipentatope elements are shown in Figure~\ref{coupled_elements_cubic_pyramid}.

\begin{figure}[h!]
\begin{center}
\includegraphics[width=9cm]{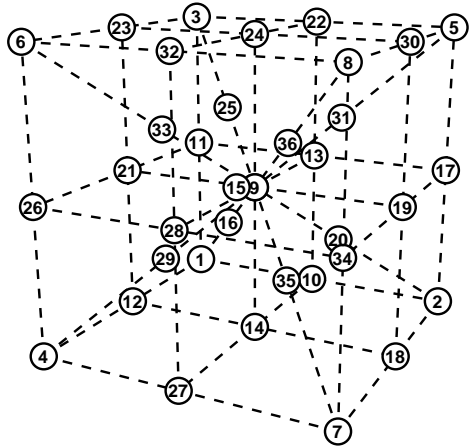}
\end{center}
\caption{The cubic pyramid $K^{\ast}$ and the full set of vertices for the subdivision operator $\mathcal{D}K^{\ast}$.}
\label{cubic_pyramid_with_vertices}
\end{figure}

\begin{figure}[h!]
\begin{center}
\includegraphics[width=9cm]{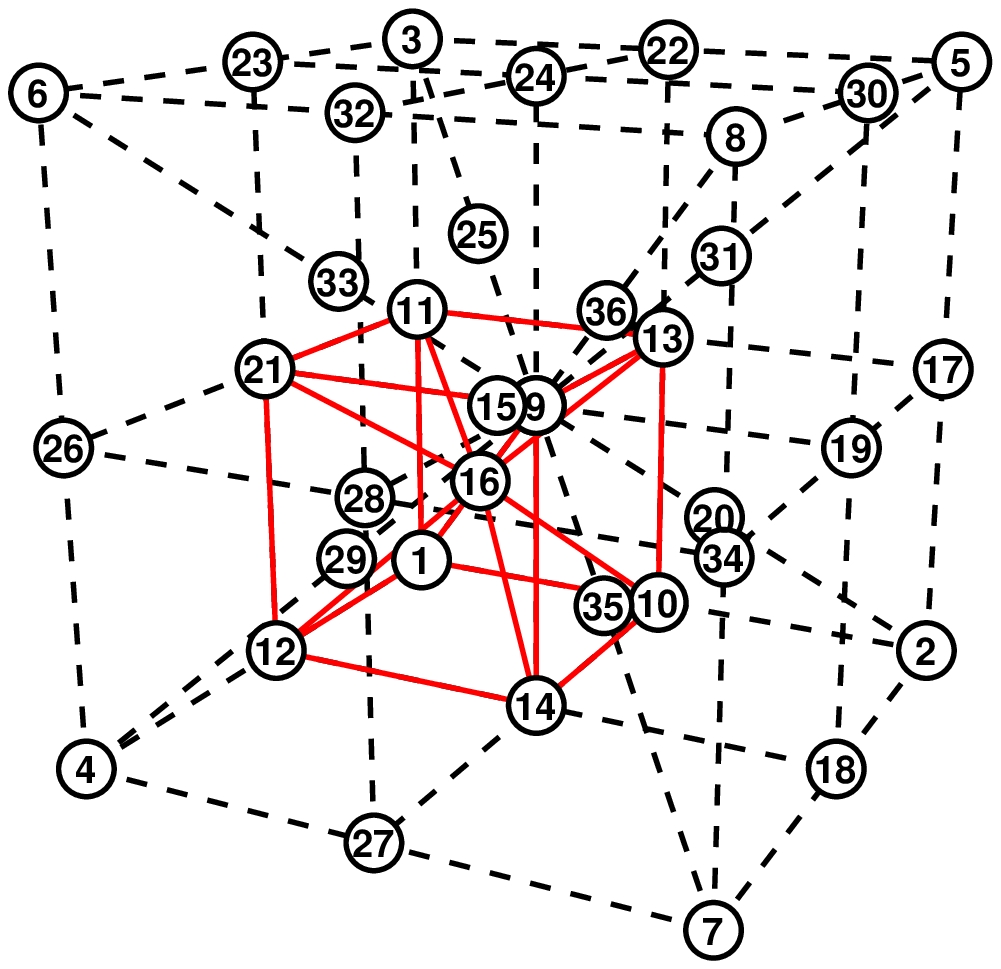}
\includegraphics[width=9cm]{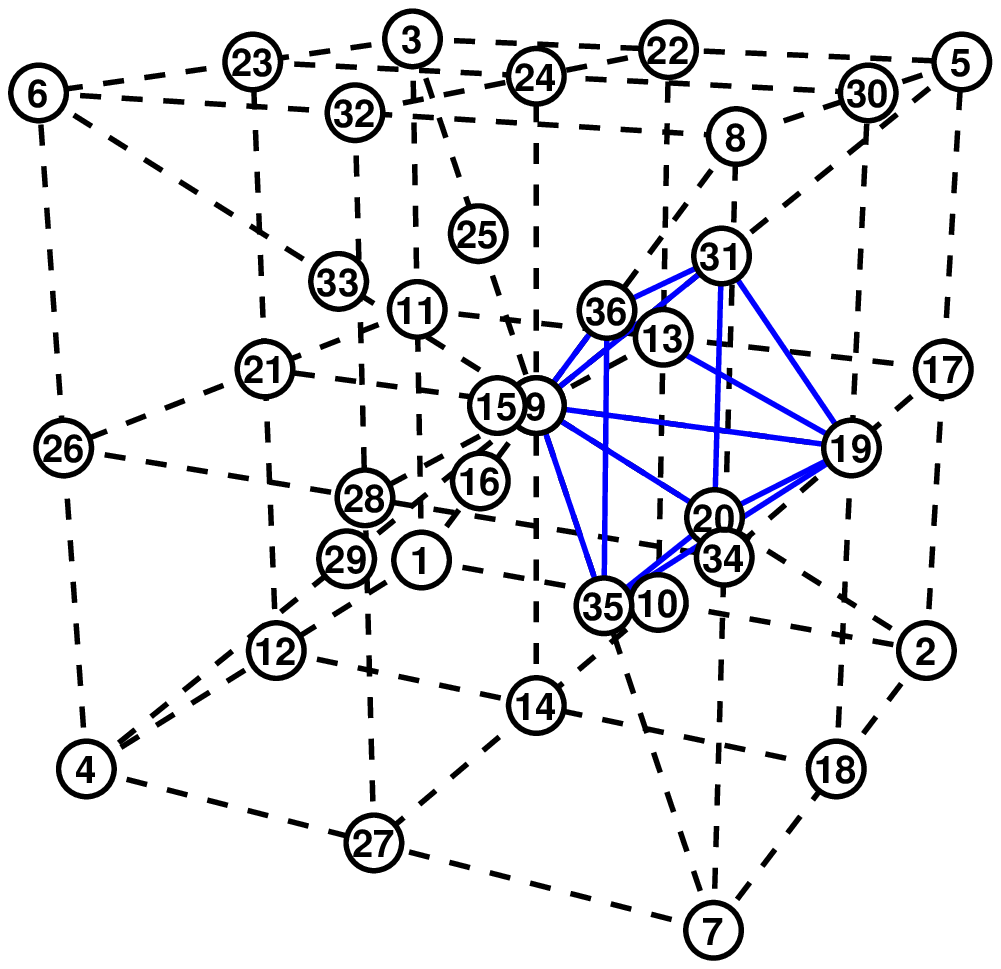}
\end{center}
\caption{The cubic pyramid element $\Ktwo_1$ and the bipentatope element $\Pone_{13}$ of the decomposition $\mathcal{D} K^{\ast}$.}
\label{sample_elements_cubic_pyramid}
\end{figure}

\begin{figure}[h!]
\begin{center}
\includegraphics[width=9cm]{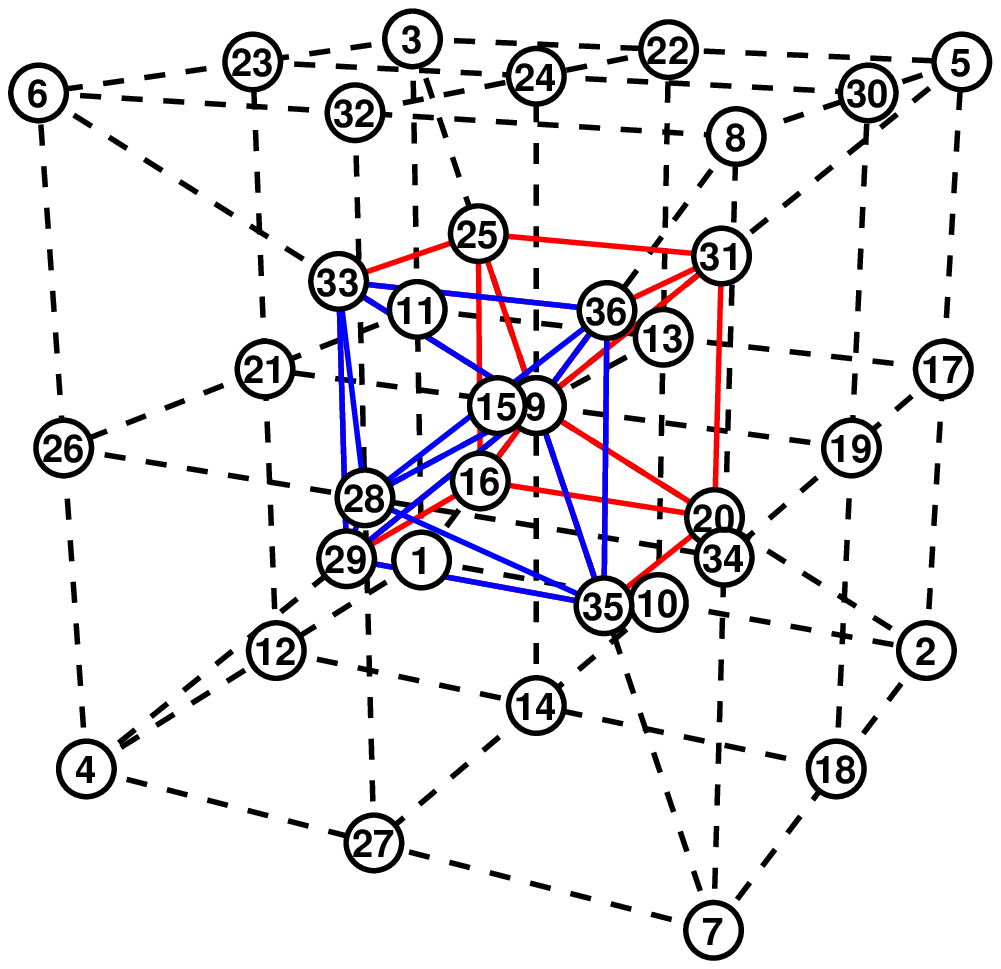}
\includegraphics[width=9cm]{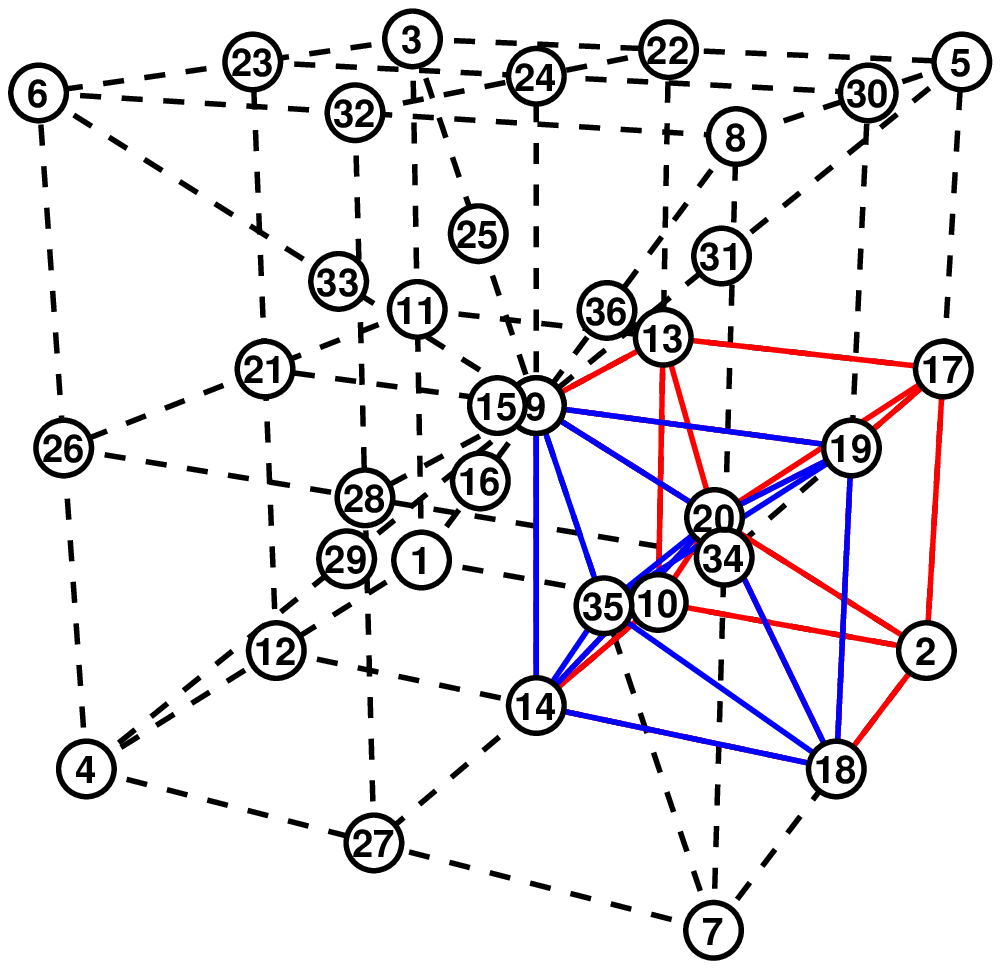}
\end{center}
\caption{Couplings between cubic pyramid and bipentatope elements: $\Ktwo_{10}$ and $\Pone_{14}$ (top) and $\Ktwo_2$ and $\Pone_{10}$ (bottom).} \label{coupled_elements_cubic_pyramid}
\end{figure}

Next, the operator $\mathcal{M}$ can be defined as follows
\begin{align*}
\mathcal{M} P^*=\bigg\{ &\Ptwo_1[p_{5},p_{10},p_{13},p_{16},p_{18},p_{20}],\; \Ptwo_2[p_{6},p_{11},p_{14},p_{17},p_{19},p_{20}],\; \Ptwo_3[p_{8},p_{10},p_{13},p_{16},p_{18},p_{20}], \\
&\Ptwo_4[p_{8},p_{11},p_{14},p_{17},p_{19},p_{20}], \;
\Ptwo_5[p_{8},p_{16},p_{17},p_{18},p_{19},p_{20}],\;
\Ptwo_6[p_{8},p_{13},p_{14},p_{16},p_{17},p_{20}], \\
&\Ptwo_7[p_{8},p_{10},p_{11},p_{18},p_{19},p_{20}], \;
\Ptwo_8[p_{8},p_{10},p_{11},p_{13},p_{14},p_{20}], \;
\Ptwo_9[p_{3},p_{8},p_{12},p_{15},p_{16},p_{17}],\\ &\Ptwo_{10}[p_{4},p_{8},p_{9},p_{15},p_{18},p_{19}], \;
\Ptwo_{11}[p_{1},p_{7},p_{8},p_{9},p_{10},p_{11}], \;
\Ptwo_{12}[p_{2},p_{7},p_{8},p_{12},p_{13},p_{14}], \\
&\Ptwo_{13}[p_{8},p_{15},p_{16},p_{17},p_{18},p_{19}], \;
\Ptwo_{14}[p_{8},p_{12},p_{13},p_{14},p_{16},p_{17}], \;
\Ptwo_{15}[p_{7},p_{8},p_{10},p_{11},p_{13},p_{14}],\\ 
& \Ptwo_{16}[p_{8},p_{9},p_{10},p_{11},p_{18},p_{19}] \bigg\}.
\end{align*}
All the elements $\Ptwo_{j}$, $j = 1, 2, \ldots, 16$, are bipentatopes. 
The first six vertices (above) $p_1, p_2, \ldots, p_6$ are previously defined in the definition of $P^{\ast}$. 
In addition, the remaining vertices $p_7, \ldots, p_{20}$ can be defined as follows
\begin{alignat*}{4}
    &p_7 = \tfrac{1}{2}(p_1+p_2), \qquad &&p_8 = \tfrac{1}{2}(p_1+p_{3}), \qquad &&p_9 = \tfrac{1}{2}(p_1+p_4), \qquad &&p_{10} = \tfrac{1}{2}(p_1+p_5), \\
    &p_{11} = \tfrac{1}{2}(p_1+p_6), \;
    &&p_{12} = \tfrac{1}{2}(p_2+p_3), \;
    &&p_{13} = \tfrac{1}{2}(p_2+p_5), \;
    &&p_{14} = \tfrac{1}{2}(p_2+p_6), \\
    &p_{15} = \tfrac{1}{2}(p_3+p_4), \;
    &&p_{16} = \tfrac{1}{2}(p_3+p_5), \;
    &&p_{17} = \tfrac{1}{2}(p_3+p_6), \;
    &&p_{18} = \tfrac{1}{2}(p_4+p_5), \\
    &p_{19} = \tfrac{1}{2}(p_4+p_6),
    &&p_{20} = \tfrac{1}{2}(p_5+p_6).
\end{alignat*}

The full set of vertices is illustrated in Figure~\ref{bipentatope_with_vertices}.
In addition, a few of the individual elements of $\mathcal{M} P^{\ast}$ are shown in Figure~\ref{sample_elements_bipentatope}. 

\begin{figure}[h!]
\begin{center}
\includegraphics[trim=0cm 0.75cm 0cm 0cm, clip, width=9cm]{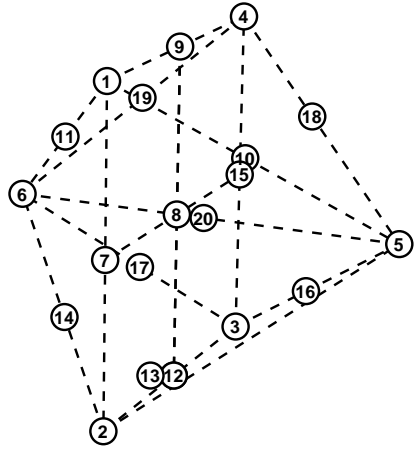}
\end{center}
\caption{The bipentatope $P^{\ast}$ and the full set of vertices for the subdivision operator $\mathcal{M}P^{\ast}$.}
\label{bipentatope_with_vertices}
\end{figure}

\begin{figure}[h!]
\begin{center}
\includegraphics[width=9cm]{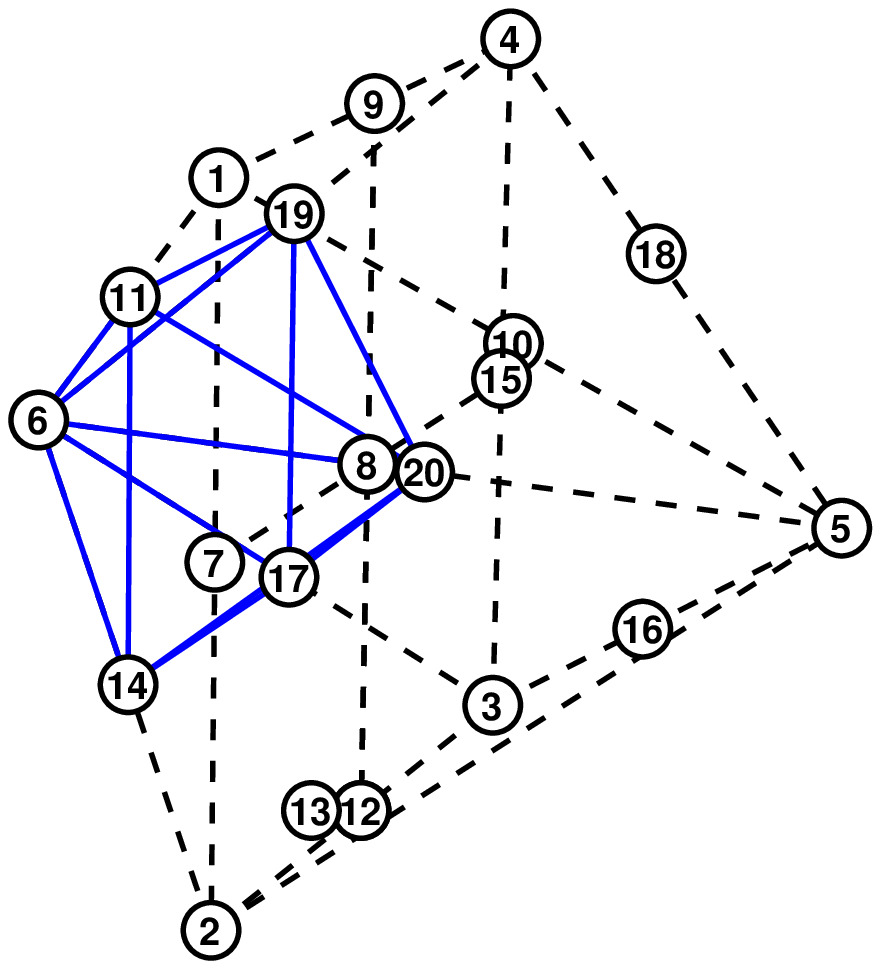}
\includegraphics[width=9cm]{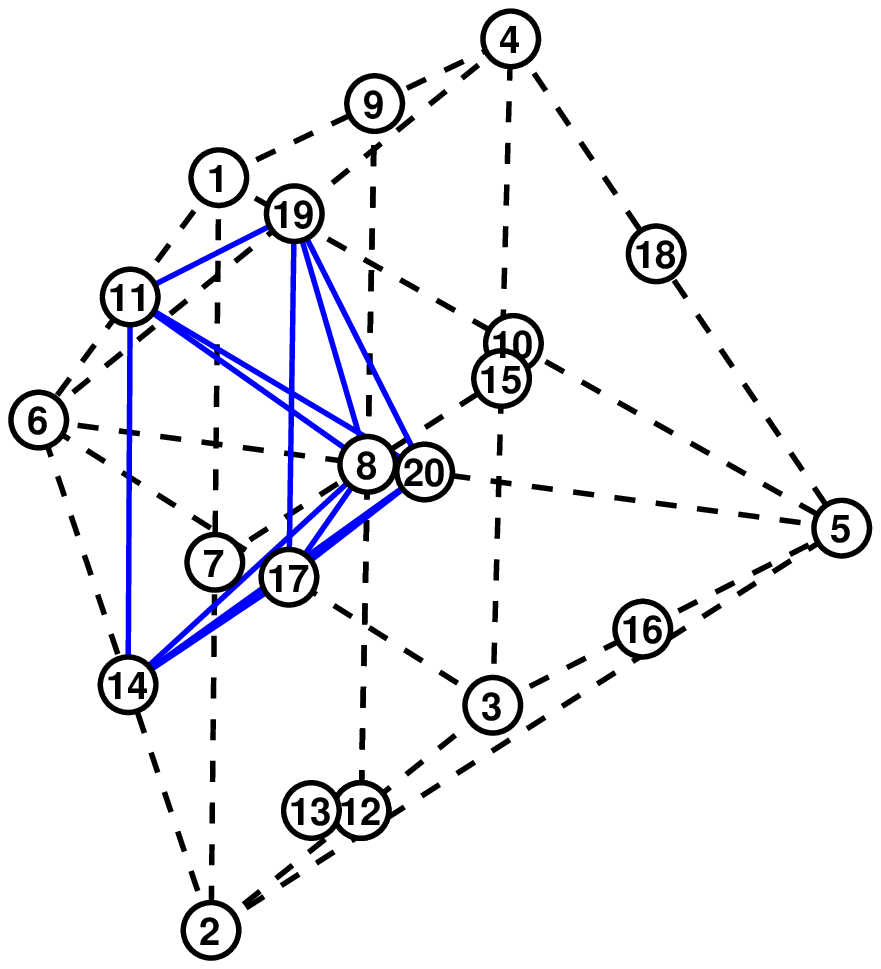}
\includegraphics[width=9cm]{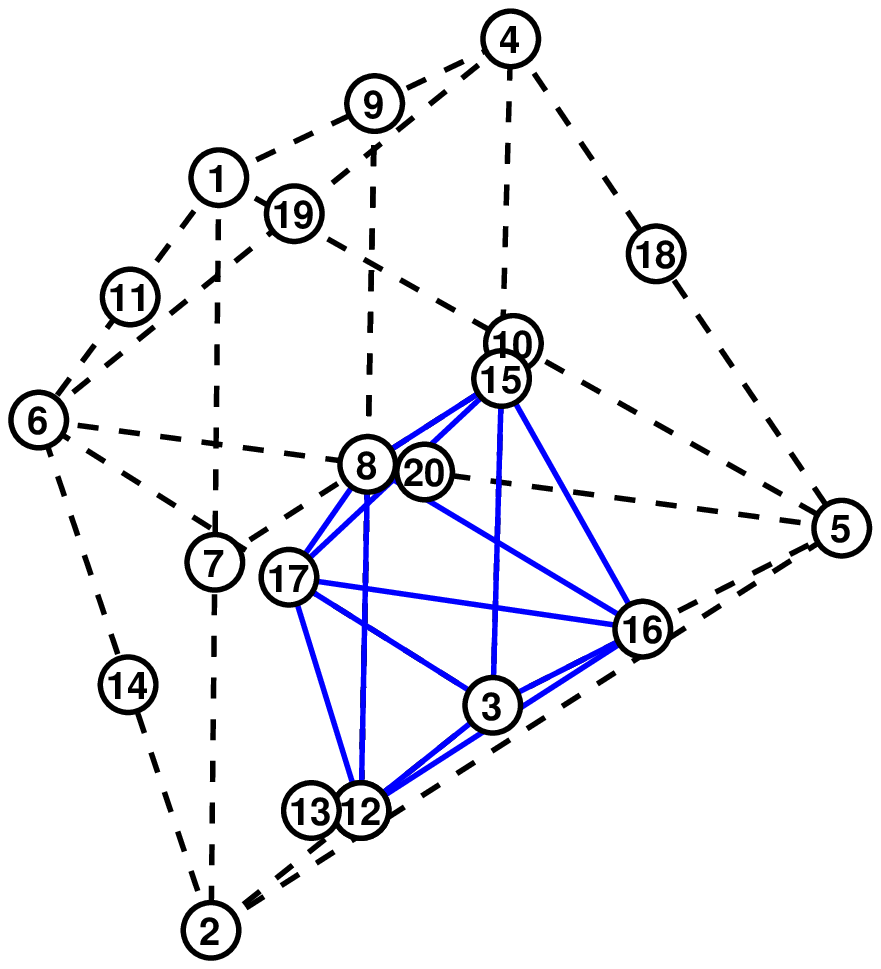}
\end{center}
\caption{The elements $\Ptwo_{2}$, $\Ptwo_{4}$, and $\Ptwo_{9}$ of the decomposition $\mathcal{M} P^{\ast}$.}\label{sample_elements_bipentatope}
\end{figure}

We conclude this section by summarizing the refinement procedure for obtaining hierarchical triangulations below:
\begin{align*}
\tau_0 = \mathcal{R}\Omega,\quad \tau_0=\check{\tau}_0\cup\hat{\tau}_0;
\end{align*}
\begin{align*}
\hat{\tau}_0^* = \mathcal{B}\hat{\tau}_0,
\quad \tau_0^*=\check{\tau}_0\cup\hat{\tau}_0^*; 
\end{align*}
\begin{align*}
\check{\tau}_1=\mathcal{E}\check{\tau}_0,\;
\hat{\tau}_1=\mathcal{L}\hat{\tau}_0^*,
\quad \tau_1=\check{\tau}_1\cup\hat{\tau}_1; 
\end{align*}
\begin{equation}\label{t1}
\check{\tau}_2=\mathcal{E}\check{\tau}_1,\;
\hat{\tau}_2=\mathcal{L}\hat{\tau}_1,
\quad \tau_2=\check{\tau}_2\cup\hat{\tau}_2;
\end{equation}
\begin{align*}
\check{\tau}_{k+1}=\mathcal{E}^{k} \check{\tau}_{1},\;
\hat{\tau}_{k+1}=\mathcal{L}^{k} \hat{\tau}_{1},
\quad \tau_{k+1}=\check{\tau}_{k+1}\cup\hat{\tau}_{k+1}.
\end{align*}
In accordance with ($\ref{t1}$), we can explicitly define the partition operator $\mathcal{H} :\tau_1\rightarrow\tau_2$ as follows
\begin{align*}
\mathcal{H} X=
\begin{cases}
  \mathcal{E}X, & \mbox{if } X \;\mbox{is a tesseract}\\
  \mathcal{L}X, & \mbox{if } X \; \mbox{is a cubic pyramid or bipentatope.}
\end{cases}
\end{align*}
Then
\begin{align*}
\tau_{k+1}=\mathcal{H} ^k\tau_{1}.
\end{align*}
Thus, we have a conforming sequence of successive hybrid finite element triangulations $\left\{\tau_{k}\right\}$.

\pagebreak
\clearpage

\section{Theoretical results}

In this section, we summarize the important theoretical results that govern the hybrid triangulations introduced in the previous section.

\begin{pro}\label{pro}
The hypercube $T^n$ can be divided into $2n$ hypercubic pyramids $\Kone^n_i$ so that all edges of the $i$-th pyramid have the same length.
\end{pro}
\begin{thm}\label{uni}
The four-dimensional space is the only Euclidean space that satisfies Property~1.
\end{thm}

\begin{proof}
For the sake of simplicity, we suppose that
\begin{align*}
T^n=\left[
t_1^n(-a,-a,\ldots,-a),\ldots
t_{2^n}^n(a,a,\ldots,a)\right]
\end{align*}
has all edges parallel to the coordinate axes. Then, Property~\ref{pro} is valid iff
\begin{align*}
\frac{1}{2}h[T^n]=l[T^n],
\end{align*}
where $l[T^n]$ is the edge length of $T^n$. The latter equality leads to
\begin{align*}
\frac{1}{2}\sqrt{\sum_{k=1}^n \left(2a\right)^2}= 2a\Rightarrow \sqrt{na^2}=2a,
\end{align*}
which is possible iff $n=4$. \qed
\end{proof}

The result of Theorem~\ref{uni} is that the four-dimensional space is unique among all Euclidean spaces.



\begin{thm}\label{inv}
Consider the cubic pyramids in $\mathcal{D} K^{\ast}$.
All such pyramids are congruent to the reference pyramid, i.e.~$\Ktwo_i\subset[K^*],$ $i=1,2,\ldots,10$.
\end{thm}

\begin{proof}
For all elements $\Ktwo_i$ $i=1,2,\ldots10,$ of  $\mathcal{D} K^*$ it is possible to verify through direct calculation that:

\begin{enumerate}[label=(\roman*)]
\item all edges of $\Ktwo_i$ are equal to $1$;
\item the base of $\Ktwo_i$ is a cube with edges equal to $1$;
\item the degeneracy measure of $\Ktwo_i$ is $\delta(\Ktwo_i)= 4.18154$;
\end{enumerate}
from which congruency with the reference pyramid follows. \qed
\end{proof}

\begin{thm}\label{con}
Consider the bipentatopes in $\mathcal{D} K^{\ast}$.
All such bipentatopes are congruent.
\end{thm}

\begin{proof}
Without loss of generality, we consider demonstrating this property for two bipentatopes chosen from $\mathcal{D} K^*$, with the claim following from repeating this procedure for all pairwise combinations.  With this in mind, we take our two bipentatopes chosen from $\mathcal{D} K^*$ to be $\Pone_{10}$ and $\Pone_{14}$, see Figure \ref{coupled_elements_cubic_pyramid}.
Any bipentatope can be partitioned into
two pentatopes as shown in Figure~\ref{split_bipentatope}.
We hence begin by dividing the bipentatopes $\Pone_{10}$ and $\Pone_{14}$ into pentatopes whence
\begin{align*}
\Pone_{10}=\hat S_{10}
\left[p_{15},p_{18},p_{19},p_{20},p_{35}\right]\cup
\check{S}_{10}\left[p_{14},p_{15},p_{18},p_{20},p_{35}\right], \\
\Pone_{14}=\hat S_{14}
\left[p_{15},p_{28},p_{29},p_{33},p_{36}\right]\cup
\check{S}_{14}\left[p_{15},p_{28},p_{29},p_{35},p_{36}\right].
\end{align*}

Let the generic affine transformations of the pentatopes $\hat S_{10}$ and $\hat S_{14}$ be
\begin{align*}
F_{10}\;:\;\hat S_{10}=A_{10}\mathcal{S}^*+B_{10},\quad
F_{14}\;:\;\hat S_{14}=A_{14}\mathcal{S}^*+B_{14}.
\end{align*}
We map $\hat S_{14}$ onto $\hat S_{10}$ by
\begin{align*}
\hat S_{10}=F_{10}\left(F_{14}^{-1}\left(\hat S_{14}\right)\right)
\Leftrightarrow
F_{14,10}\;:\;\hat S_{10}=A_{14,10}\hat S_{14}+B_{10}-A_{14,10}B_{14},
\end{align*}
where $A_{14,10}$ is the transitional matrix from $\hat S_{14}$ to
$\hat S_{10}$ and $F_{14,10}=F_{10}\circ F_{14}^{-1}$.
We emphasize that $p_{14}=F_{14,10}(p_{35})$, which assures
that
$
\Pone_{10}=F_{14,10}\left(\Pone_{14}\right)$.
The claim that $\Pone_{10}\cong \Pone_{14}$ follows from the fact that the transition matrix
\begin{align*}
A_{14,10}=\left(
\begin{array}{cccc}
 -1 & 0 & 0 & 0 \\
 0 & 0 & 1 & 0 \\
 0 & -1 & 0 & 0 \\
 0 & 0 & 0 & -1 \\
\end{array}
\right)
\end{align*}
is orthogonal. \qed
\end{proof}

\begin{rem}

Let all bipentatopes from $\mathcal{D} K^*$ be partitioned into
two pentatopes as it is shown in Figure~\ref{split_bipentatope}.
Then all of them belong to $[\hat S_{10}]$ and
$\delta\left([\hat S_{10}]\right)=5.$ Moreover,
the elements of $[\hat S_{10}]$ are invariant. For more details concerning this class of pentatopes, we refer the reader to~\cite{petrov2020properties}.
\end{rem}


\begin{thm}\label{subd}
Consider the bipentatopes in $\mathcal{M} P^{\ast}$.
All such bipentatopes are congruent to the reference bipentatope, i.e.~$\Ptwo_j\subset[P^{\ast}],$ $j=1,2,\ldots,16$.
\end{thm}

\begin{proof}
Firstly, we prove that all bipentatopes in $\mathcal{M} P^*$ belong to the same class.
We randomly choose two bipentatopes
$\Ptwo_{2}$ and $\Ptwo_{4}$ (see Figure~\ref{sample_elements_bipentatope}) from $\mathcal{M} P^*$.
The bipentatope $\Ptwo_{2}$ is obtained by the coupling of the pentatopes
\begin{align*}
\hat S_{2}
\left[p_{6},p_{11},p_{14},p_{17},p_{20}\right] \quad \mbox{and} \quad
\check{S}_{2}\left[p_{6},p_{11},p_{17},p_{19},p_{20}\right].
\end{align*}
The second bipentatope is partitioned as follows
\begin{align*}
\Ptwo_{4}=\hat S_{4}
\left[p_{8},p_{11},p_{14},p_{17},p_{20}\right]\cup
\check{S}_{4}\left[p_{8},p_{11},p_{17},p_{19},p_{20}\right].
\end{align*}
The pentatope $\hat S_{2}$ is mapped onto $\hat S_{4}$ by the affine transformation
\begin{equation*}
\hat S_{4}=F_{4}\left(F_{2}^{-1}\left(\hat S_{2}\right)\right),
\end{equation*}
where
\begin{align*}
F_{2}\;:\;\hat S_{2}=A_{2}\mathcal{S} ^*+B_{2}, \quad
F_{4}\;:\;\hat S_{4}=A_{4}\mathcal{S} ^*+B_{4}.
\end{align*}
We present the map $F_{2,4}=F_{4}\circ F_{2}^{-1}$
in a matrix form by
\begin{align*}
F_{2,4}\;:\;\hat S_{4}=A_{2,4}\hat S_{2}+B_{4}-A_{2,4}B_{4}.
\end{align*}
The transitional matrix
\begin{align*}
A_{2,4}=\left(
\begin{array}{cccc}
 1 & 0 & 0 & 0 \\
 0 & 0 & 0 & -1 \\
 0 & 0 & 1 & 0 \\
 0 & -1 & 0 & 0 \\
\end{array}
\right)
\end{align*}
is orthogonal. Additionally, $F_{2,4}$ keeps the facet $\Ptwo_{2,-1,-6}$ invariant, which assures $\Ptwo_{4}=F_{2,4}\left(\Ptwo_{2}\right)$. Therefore, $\Ptwo_{2}\cong \Ptwo_{4}$.

On the other hand, all edges of the bipentatopes have the same length equal to $\frac{1}{2}$, i.e. they are identical.

It remains to prove that $\Ptwo_{2}$ belongs to $[P^*]$.
This follows directly from
$A_{2,*}=2I$, where $A_{2,*}$ is the transitional matrix from $\Ptwo_{2}$ to $P^*$ and $I$ is the identity matrix.
\qed
\end{proof}



\section{Quadrature rules for the cubic pyramid}

\subsection{Methodology}

In this section, we consider the problem of approximating the integral of a function $f(\boldsymbol{x})$ on $K^{\ast}$, the reference cubic pyramid.  This can be effectively accomplished via a quadrature rule whence
\begin{equation}\label{eq:moment}
    \int_{-1}^{0} \int_{x_4}^{-x_4} \int_{x_4}^{-x_4} \int_{x_4}^{-x_4} f(\boldsymbol{x}) \,\mathrm{d}x_1 \, \mathrm{d}x_2 \, \mathrm{d}x_3 \, \mathrm{d}x_4 \approx \sum_{j=1}^N \omega_j f(\boldsymbol{x}_j),
\end{equation}
where $\{\omega_j\}$ are a set of $N$ weights and $\{\boldsymbol{x}_j\}$ are an associated set of $N$ points---known as abscissa---which together specify a \emph{quadrature rule}.  If a rule is capable of exactly computing the integrals of a generic constant and all polynomials in the set
\[
 \Xi(p) = \{ \psi_{ijkq}(\boldsymbol{x}) \mid  0 < i + j + k + q \le p \text{ and } i, j, k, q \ge 0 \},
\]
where $\psi_{ijkq}(\boldsymbol{x})$ is as given in Eq.~\eqref{eq:cubpb}, then the rule is said to be of strength $p$.  When defining the abscissa of a rule, it is typical to constrain them to be strictly inside of our domain and be arranged symmetrically.  This ensures that we never evaluate $f(\boldsymbol{x})$ outside of its domain.  In addition, to reduce the likelihood of catastrophic cancellation it is also customary to require the weights to be positive.  Evidently if the rule is to integrate a generic constant mode correctly it must necessarily be the case that
\[
 \sum_{j=1}^N \omega_j = 2.
\]

Substituting the polynomials from our set into Eq. \eqref{eq:moment} and enforcing equality we require for all $\psi \in \Xi(p)$ that 
\[
 \int_{-1}^{0} \int_{x_4}^{-x_4} \int_{x_4}^{-x_4} \int_{x_4}^{-x_4} \psi(\boldsymbol{x}) \,\mathrm{d}x_1 \, \mathrm{d}x_2 \, \mathrm{d}x_3 \, \mathrm{d}x_4 = 0 = \sum_{j=1}^N \omega_j \psi(\boldsymbol{x}_j),
\]
where in evaluating the integral we have exploited the orthogonality relationship associated with the polynomials in our basis.  This can be regarded as a non-linear least squares problem in which we have $|\Xi(p)| + 1$ equations and $5N$ unknowns.  However, unlike a typical least squares problem, it is one where we seek a solution with zero residual.   Assuming the abscissa are known we remark that the associated weights may be determined through the solution of a \emph{linear} least squares problem.  This enables us to reduce the number of degrees of freedom in our non-linear problem to $4N$.

In order to enforce symmetry, it is advantageous to decompose our $N$ abscissa into symmetry orbits. The symmetry orbits of the cubic pyramid are similar to those of the hexahedron, with an additional parameter that accounts for the extent of the pyramid in the fourth dimension. Therefore, one may follow the procedure for generating orbits for the hexahedron (see~\cite{witherden2015identification}), and enumerate all symmetry orbits. Thereafter, upon adding a parameter to account for the fourth dimension, the symmetry orbits of the cubic pyramid are written as follows
\begin{equation*}
\begin{aligned}[c]
S_1 (\delta) &= (0,0,0,\delta),\\
S_2 (\alpha,\delta) &= \chi(\alpha,0,0,\delta),\\
S_3 (\alpha,\delta) &= \chi(\alpha,\alpha,\alpha,\delta), \\
S_4 (\alpha,\delta) &= \chi(\alpha,\alpha,0,\delta), \\
S_5 (\alpha,\beta,\delta) &= \chi(\alpha,\beta,0,\delta), \\
S_6 (\alpha,\beta,\delta) &= \chi(\alpha,\alpha,\beta,\delta), \\
S_7 (\alpha,\beta,\gamma,\delta) &= \chi(\alpha,\beta,\gamma,\delta),
\end{aligned}
\qquad \qquad
\begin{aligned}[c]
|S_1| &= 1, \\
|S_2| &= 6, \\
|S_3| &= 8, \\
|S_4| &= 12, \\
|S_5| &= 24, \\
|S_6| &= 24, \\
|S_7| &= 48,
\end{aligned}
\end{equation*}
subject to the constraints $0 < \alpha,\beta,\gamma \leq - \delta$ and $-1 \leq \delta \leq 0$. Here, the operator $\chi$ returns all possible (unique) signed permutations of the input arguments. 

Now, let us provide an example of how these symmetry orbits are used. If $N = 156$, then one possible arrangement of points is four $S_1$ orbits, two $S_2$ orbits, one $S_3$ orbit, three $S_4$ orbits, two $S_5$ orbits, and one $S_7$ orbit whence
\[
 N = 4|S_1| + 2|S_2| + |S_3| + 3|S_4| + 2|S_5| + |S_7| = 156,
\]
as required.  Inspecting the number of arguments for each of the orbits and summing we find that this decomposition has a total of $26$ degrees of freedom.  Together these fix the locations of all of our abscissa and hence become the parameters in our non-linear least squares problem.  We note here that this represents a substantial decrease compared with an asymmetric rule which has some $4N = 624$ degrees of freedom.  Of course, for a given $N$ there are typically a very large number of distinct decompositions into symmetry orbits.  For our example of $N = 156$ there are some $8518$ unique decompositions.

Our overall goal when identifying quadrature rules is to find one with the minimum value of $N$ for a particular strength $p$.  Unfortunately, given a $p$ it is not possible to determine $N$ theoretically; instead it must be treated as a hyper-parameter to the overall optimisation process.  Similarly, once given an $N$ it is, in general, not possible to know \emph{a priori} which particular symmetric decomposition will yield a rule--if any.  Thus, it is also necessary to treat the specific decomposition as a hyper-parameter, too.

We have implemented the aforementioned system into the open source symmetric quadrature rule package Polyquad \cite{witherden2015identification}.  Running this modified version of Polyquad on a workstation, we have identified symmetric quadrature rules on the cubic pyramid for $2 \le p \le 12$.  The total number of points required for these rules can be seen in Table \ref{tab:rules}. 
The abscissa for all of our rules are inside of our reference pyramid, and feature positive weights. The rules themselves are provided in quadruple precision in the electronic supplemental material. An example of the $p = 5$ quadrature rule is provided in Table~\ref{tab:example}.

\begin{table}[h!]
    \centering
    \caption{Number of points $N$ required for a rule of strength $p$ inside of the cubic pyramid.}
    \begin{tabular}{r|rrrrrrrrrrr} \toprule
     $p$ &  2 & 3 & 4 & 5 & 6 & 7 & 8 & 9 & 10 & 11 & 12 \\
     $N$    &  7 & 10 & 21 & 29 & 50 & 65 & 114 & 163 & 251 & 323 & 552 \\
    \bottomrule
    \end{tabular}
    \label{tab:rules}
\end{table}

\begin{table}[h!]
    \centering
    \caption{The $p = 5$ quadrature rule on the cubic pyramid $K^{\ast}$ with $N = 29$ points. The remaining rules are provided in the electronic supplemental material.}
    \begin{tabular}{r|l|l} \toprule
    Orbit & Abscissa & Weight \\ 
    \hline
    $S_1$ & -0.32846339581882957902277574789024 & 0.084184481885656624083314191434215
    \\
    \hline
    $S_2$ & 0.59942266549153142521578075246447 & 0.11340384654119720346152792514806 \\
    & -0.74912930534502345701037962143888 & \\
    \hline
    $S_2$ & 0.76070275297226284532398968584027 & 0.089004419809134534555110633350815 \\
    & -0.95754362470717092670167055158225 & \\
    \hline
    $S_3$ & 0.43384835705078179649267535015814 & 0.033983660519520996993103370708523 \\
    & -0.5766200661792940506119682920933 & \\
    \hline
    $S_3$ & 0.69719510426558290268820473081072 & 0.05368707948202312148400343648805 \\
    & -0.91852788682445245487884982935091 & \\
    \bottomrule
    \end{tabular}
    \label{tab:example}
\end{table}

\subsection{Numerical experiments}

The objective of this section is to numerically validate a subset of the quadrature rules in Table~\ref{tab:rules}. Towards this end, we performed a series of numerical experiments on the rules with strengths $p = 2, 3, \ldots, 9$. These rules are capable of exactly integrating certain monomial functions by construction, as well as approximately integrating certain transcendental functions. In what follows, we limit our focus to: a) weighted combinations of monomial functions, and b) analytic transcendental functions. In addition, we execute all of the numerical experiments in quadruple precision arithmetic in order to demonstrate the full precision of the rules.

\subsubsection{Polynomial integration}

In this section, we evaluate the ability of our quadrature rules to integrate a weighted combination of monomial functions on the reference cubic pyramid $K^{\ast}$. More specifically, we consider the following polynomial function of degree~$m$
\begin{align*}
    f_{\text{poly}} \left(\boldsymbol{x}; m\right) = \sum_{r = 0}^{m} \; \sum_{s=0}^{m-r} \; \sum_{t=0}^{m-r-s} \; \sum_{v=0}^{m-r-s-t} C_{rstv} \, x_{1}^{r} x_{2}^{s} x_{3}^{t} x_{4}^{v},
\end{align*}
where the constants $C_{rstv}$ are given by the following formula
\begin{align*}
    C_{rstv} = 24\frac{\left(r+1\right)\left(s+1\right)\left(t+1\right)\left(v+1\right)}{\left(m+1\right)\left(m+2\right)\left(m+3\right)\left(m+4\right)}.
\end{align*}
One may observe that this function contains all possible distinct monomials of degree $m$. It turns out that the integral of this function on the domain $K^{\ast}$ can be computed exactly as follows
\begin{align*}
    J_{\infty} \left(\boldsymbol{x}; m\right) = \int_{K^{\ast}} f_{\text{poly}} \left(\boldsymbol{x}\right) d\boldsymbol{x} = \sum_{r = 0}^{m} \; \sum_{s=0}^{m-r} \; \sum_{t=0}^{m-r-s} \; \sum_{v=0}^{m-r-s-t} C_{rstv} I_{rstv},
\end{align*}
where
\begin{align*}
    I_{rstv} = \int_{K^{\ast}} x_{1}^{r} x_{2}^{s} x_{3}^{t} x_{4}^{v} \, d\boldsymbol{x} = \frac{\left(1+ \left(-1\right)^{r}\right)\left(1+ \left(-1\right)^{s}\right)(1+ \left(-1\right)^{t}) \left(-1\right)^{v}}{\left(1+r\right)\left(1+s\right)\left(1+t\right)\left(r+s+t+v+4\right)},
\end{align*}
is the exact integral of each monomial. 

The approximate integral of $f_{\text{poly}}$ can be computed by evaluating the following quadrature formula
\begin{align*}
    J_{p} = \sum_{j=1}^{N} \omega_{j} f_{\text{poly}} \left(\boldsymbol{x}_j\right),
\end{align*}
where the number of points $N = N\left(p\right)$, weights $\omega_j = \omega_j \left(p\right)$, and abscissa $\boldsymbol{x}_j = \boldsymbol{x}_j \left(p\right)$ are functions of the quadrature rule strength $p$. Evidently, we obtain exact integration $J_{p} = J_{\infty}$ when $p \geq m$, and approximate integration $J_{p} \approx J_{\infty}$ when $p < m$.

Figure~\ref{polynomial_error} shows the quadrature error produced by integrating $f_{\text{poly}}$ for different values of $m$ and $p$. In accordance with expectations, we obtain exact integration whenever $p \geq m$, to within machine precision ($10^{-32}$).
\begin{figure}[h!]
\begin{center}
\includegraphics[width=9.5cm]{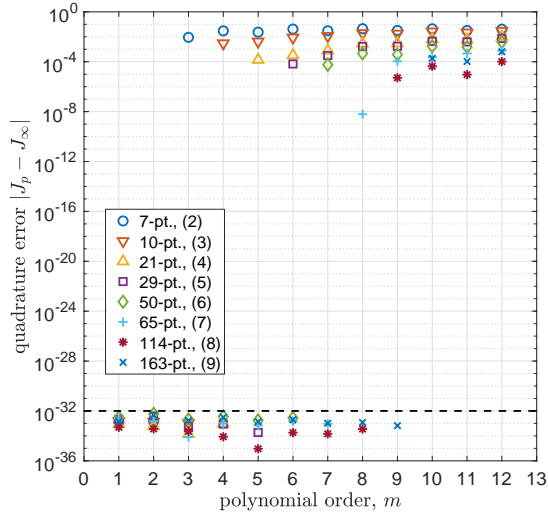}
\end{center}
\caption{Absolute error in the numerical integration of $f_{\text{poly}}$ over the reference cubic pyramid $K^{\ast}$. The dashed line marks the threshold of machine precision.}\label{polynomial_error}
\end{figure}

\subsubsection{Transcendental integration}

In this section, we evaluate the ability of our quadrature rules to integrate a set of transcendental functions on a family of successively refined meshes. For this purpose, we consider the following functions
\begin{align*}
    f_1 \left(\boldsymbol{x}\right) &= \sin\left(\pi x_{1}^{2}\right) \sin\left(\pi x_{2}^{2}\right) \sin\left(\pi x_{3}^{2}\right) \sin\left(\pi x_{4}^{2}\right), \\[1.5ex]
    f_2 \left(\boldsymbol{x}\right) &= \exp\left(x_{1}^{2}\right) \exp\left(x_{2}^{2}\right) \exp\left(x_{3}^{2}\right) \exp\left( x_{4}^{2}\right), \\[1.5ex]
    f_3 \left(\boldsymbol{x}\right) &= \exp\left( x_1 \right) \exp\left( \tfrac{1}{2} x_2 \right) \exp\left( \tfrac{1}{3} x_3 \right) \exp\left( \tfrac{1}{4} x_4 \right).
\end{align*}
Note that the first two functions are symmetric with respect to the coordinates $x_1, \ldots, x_4$, whereas the last function is asymmetric. 

The integral of each transcendental function was computed on the domain $\Omega = \left[0,1\right]^{4}$ as follows
\begin{align*}
    J_{\infty} = \int_{\Omega} f_{\text{trans}} \left(\boldsymbol{x}\right) d \boldsymbol{x}.
\end{align*}
The analytical solution of this integral is generally unknown. As a result, the `exact' integration was carried out using a vectorized adaptive integration routine in Matlab. This routine approximates the integral of a smooth function to an arbitrary level of precision by successively subdividing the domain of integration, and leveraging nested Gauss-Kronrod quadrature rules to estimate the integration error on each subinterval. One may consult the work of Shampine~\cite{shampine2008matlab,shampine2008vectorized} for details of the Matlab implementation, and Notaris~\cite{notaris2016gauss} for a general review of Gauss-Kronrod quadrature rules. In our case, we used the adaptive Matlab routine to approximate the integrals to within machine precision.

The transcendental functions above were also integrated using the aforementioned quadrature rules on the cubic pyramid. In order to facilitate this process, the domain $\Omega$ was covered with a uniform mesh of $M^{4}$ tesseract elements where $M$ is a positive integer. Thereafter, each tesseract was subdivided into 8 cubic pyramid elements in accordance with the partition operator $\mathcal{B}$, yielding a total of $N_{e} = 8M^{4}$ cubic pyramid elements. The quadrature rules were transformed (mapped) from the reference cubic pyramid $K^{\ast}$ to the individual cubic pyramids in the mesh via a linear mapping procedure. The approximate integral of each function was then computed in the following fashion
\begin{align*}
    J_{p} = \sum_{i=1}^{N_{e}} \sum_{j=1}^{N} \omega_{j}^{K_i} f_{\text{trans}} \left(\boldsymbol{x}_{j}^{K_i} \right),
\end{align*}
where $\omega_{j}^{K_i}$ are the quadrature weights and $\boldsymbol{x}_{j}^{K_i}$ are the point locations in each element $K_i$. 

Figures~\ref{f1_error}, \ref{f2_error}, and \ref{f3_error} illustrate the quadrature errors for integrating $f_1$, $f_2$, and $f_3$ on $\Omega$ for mesh parameters $M = 1, 2, \ldots, 10$. In each case, the error converges at a rate of $h^{p+1}$ or higher. In addition, the pairs of even and odd quadrature rules converge at the same rates, e.g. rules with $p =2$ and $p= 3$ converge at a rate of $h^4$. Upon combining these insights together, we estimate that the rules will typically converge at a rate of $h^{p+2}$ for even values of $p$, and $h^{p+1}$ for odd values of $p$. The precise reason for this behavior is unknown, although it is likely due to fortunate cancellations of truncation error terms for quadrature rules with even values of $p$.

In addition, we note that increasing the number of quadrature points (by increasing $p$) does not always yield a more accurate result when  integrating transcendental functions. For example, the 50 point rule with $p = 6$ sometimes outperforms the 65 point rule with $p = 7$ (see Figure~\ref{f3_error}). This behavior is likely due to a convenient alignment between the symmetry orbits of the $p = 6$ rule and the local topology of the transcendental function $f_3$. We note that this trend is not general, as it does not hold for $f_1$ and $f_2$. In fact, in most cases the higher degree rules outperform the lower degree rules, as expected. 

\begin{figure}[h!]
\begin{center}
\includegraphics[width=9.5cm]{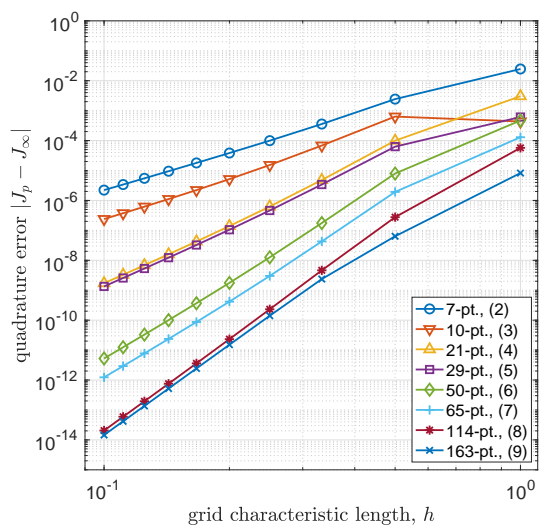}
\end{center}
\caption{Absolute value of the error in the numerical integration of $f_1$ over the domain $\Omega$ for mesh parameters $M = 1, 2, \ldots, 10$.}\label{f1_error}
\end{figure}
\begin{figure}[h!]
\begin{center}
\includegraphics[width=9.5cm]{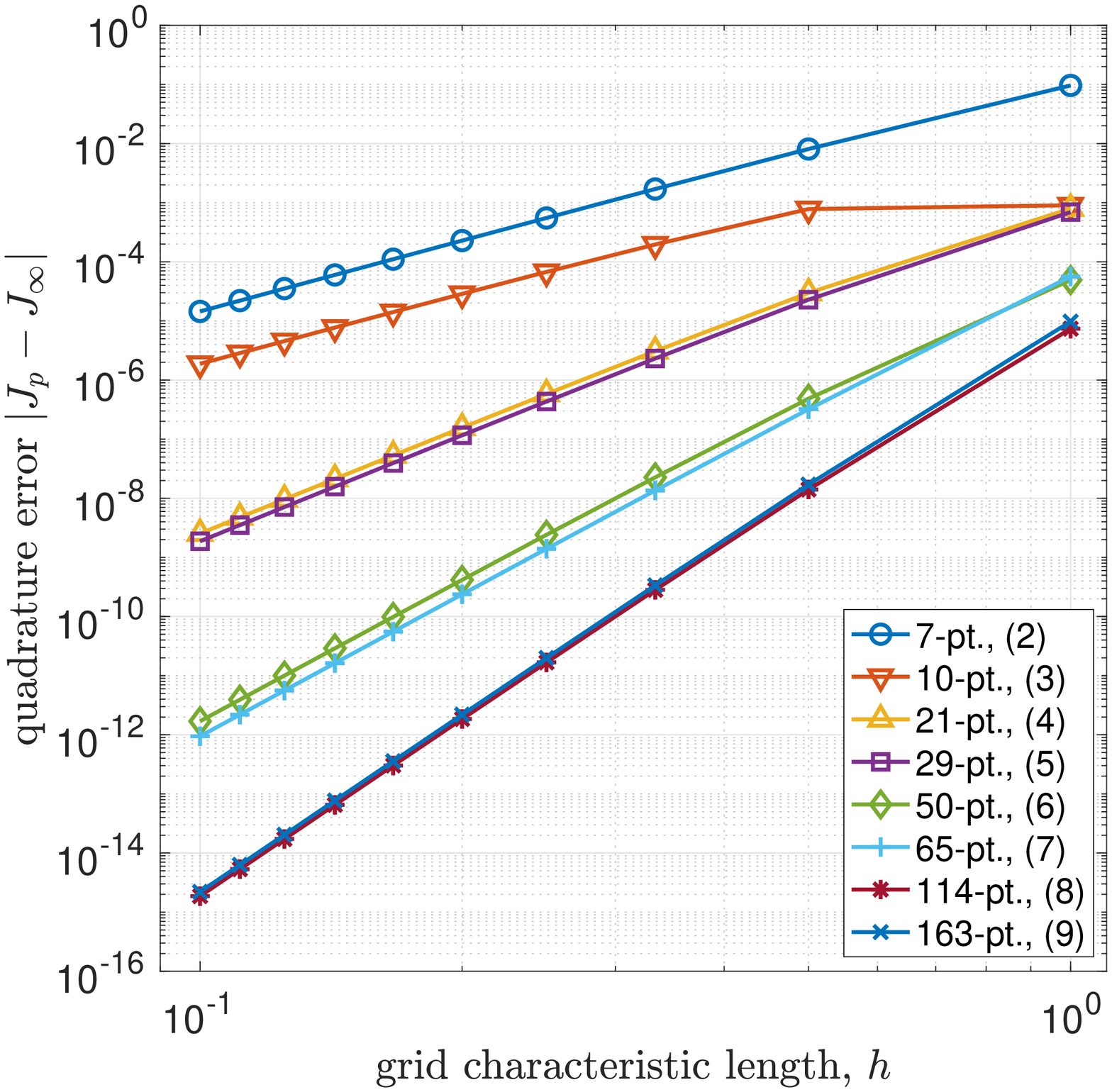}
\end{center}
\caption{Absolute value of the error in the numerical integration of $f_2$ over the domain $\Omega$ for mesh parameters $M = 1, 2, \ldots, 10$.}\label{f2_error}
\end{figure}

\pagebreak
\clearpage

\begin{figure}[h!]
\begin{center}
\includegraphics[width=9.5cm]{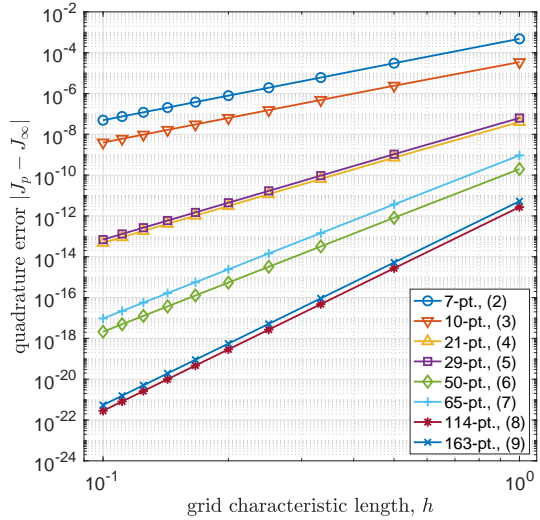}
\end{center}
\caption{Absolute value of the error in the numerical integration of $f_3$ over the domain $\Omega$ for mesh parameters $M = 1, 2, \ldots, 10$.}\label{f3_error}
\end{figure}

\section{Conclusion}

In this work, we have presented a novel refinement strategy for four-dimensional cubic pyramids.  Given a cubic pyramid our method subdivides it into a conforming set of smaller cubic pyramids and bipentatopes.  Moreover, all of the edges of each cubic pyramid have the same length, as do the edges of each bipentatope, which corresponds to an optimal configuration. Furthermore, we have also developed and evaluated a new set of polynomial quadrature rules inside of the cubic pyramid.  Together, these results open up new pathways for four-dimensional meshing for space-time finite element methods.

For the sake of completeness, we also provided a comprehensive example of a four-dimensional conformal hybrid mesh on a canonical domain. The initial mesh contains tesseract and cubic pyramid elements. The mesh is conformal because the couplings between the tesseracts and the cubic pyramids are consistent, without the need for an interface subdomain. Following the formation of the initial mesh, it is refined by subdividing the tesseract elements in the standard fashion (splitting along planes that are orthogonal to the coordinate axes), and by subdividing the cubic pyramids in the aforementioned novel fashion. We believe that this strategy can be generalized to a much broader class of hybrid meshes, including curved meshes, without significant modifications.

\section*{Declarations}

\subsubsection*{Funding}

The authors did not receive support from any organization for the submitted work.

\subsubsection*{Conflict of interest/Competing interests}

The authors have no conflicts of interest to declare that are relevant to the content of this article.

\subsubsection*{Availability of data and material}

The quadrature rules are available as electronic supplemental material.

\subsubsection*{Code availability}

The Polyquad code is available at https://github.com/PyFR/Polyquad.


\bibliographystyle{spmpsci}      
\bibliography{spacetimereferences}   

%
%

\end{document}